\documentclass[11pt, twoside]{article}
\usepackage{amssymb}
\usepackage{amsmath}
\usepackage{amsthm}
\usepackage{color}
\usepackage{mathrsfs}
\usepackage{txfonts}
\usepackage{enumerate}
\usepackage{indentfirst}
\usepackage{esint}
\usepackage{diagbox}
\usepackage{makecell}

\usepackage{anysize}

\usepackage{bbding}

\usepackage{cite}

\usepackage[title,titletoc]{appendix}

\allowdisplaybreaks

\newtheorem{theorem}{Theorem}[section]
\newtheorem{lemma}[theorem]{Lemma}
\newtheorem{corollary}[theorem]{Corollary}
\newtheorem{proposition}[theorem]{Proposition}

\theoremstyle{definition}
\newtheorem{remark}[theorem]{Remark}
\newtheorem{definition}[theorem]{Definition}

\renewcommand{\appendix}{\par
\setcounter{section}{0}%
\setcounter{subsection}{0}%
\setcounter{subsubsection}{0}%
\gdef\thesection{\@Alph\c@section}%
\gdef\thesubsection{\@Alph\c@section.\@arabic\c@subsection}%
\gdef\theHsection{\@Alph\c@section.}%
\gdef\theHsubsection{\@Alph\c@section.\@arabic\c@subsection}%
\csname appendixmore\endcsname
}

\numberwithin{equation}{section}

\begin{document}

\arraycolsep=1pt

\title{\bf\Large Characterization of Vanishing Campanato Spaces via Ball Banach Function Spaces
and Its Applications
\footnotetext{\hspace{-0.35cm} 2020 {\it
Mathematics Subject Classification}. Primary 42B35;
Secondary 42B25, 46E30.
\endgraf {\it Key words and phrases.}
BMO, Campanato space, ball Banach function space, self-improvement, vanishing mean oscillation.
\endgraf This project is partially supported by the National
Key Research and Development Program of China
(Grant No. 2020YFA0712900),
the National Natural Science Foundation of China
(Grant Nos. 12301114, 12431006, and 12371093), and
the Natural Science Foundation of Hubei Province of China
(Grant Nos. 2023AFB134 and JCZRYB202401287).
}}
\author{Xing Fu, Yoshihiro Sawano, Jin Tao\footnote{Corresponding author,
E-mail: \texttt{jintao@hubu.edu.cn}/{\color{red}\today}/Final version.}
\ and Dachun Yang
}
\date{}
\maketitle

\vspace{-0.9cm}

\begin{center}
\begin{minipage}{13cm}
{\small {\bf Abstract}\quad
In this article, the authors provide some new characterizations of
several vanishing Campanato spaces using a type of oscillation
defined within the general framework of ball Banach function spaces.
This approach yields fresh insights even in the special case of
the vanishing BMO space.
The characterization reveals a self-improvement phenomenon
inherent in vanishing Campanato spaces.
A key innovation of this approach lies in using higher-order
differences to dominate oscillations.
Instead of directly estimating these differences,
the authors achieve the domination by smoothing the function via convolution.
As additional outcomes, the authors also obtain new characterizations of
vanishing Campanato spaces in terms of higher-order differences.
Finally, the authors present several examples to show that
these vanishing Campanato spaces naturally arise in the study on
the compactness of fractional integral commutators in Morrey spaces.
}
\end{minipage}
\end{center}

\vspace{0.2cm}

\tableofcontents

\section{Introduction}
\label{s1}

Throughout the whole article, we work in $\mathbb{R}^n$ and,
unless necessary, we will not explicitly specify this underlying space.

In 1961, John and Nirenberg \cite{jn61} introduced the well-known function space
$\mathrm{BMO}$ which is one of the most significant spaces in harmonic analysis
and proves useful in so many fields such as partial differential equations
and quasiconformal mappings.
Apart from $\mathrm{BMO}$, there also exist numerous studies on its vanishing subspaces.
For instance, Sarason \cite{s75} introduced the space $\mathrm{VMO}$
to study the stationary stochastic processes satisfying the
strong mixing condition and also the algebra $H^\infty+C$;
Uchiyama \cite{u78} equivalently characterized
the compactness of Calder\'on--Zygmund commutators
using the space $\mathrm{CMO}$
announced by Neri \cite{n75};
Recently Torres and Xue \cite{tx20} introduced a middle space
$\mathrm{XMO}$, ``smaller" than $\mathrm{VMO}$
and strictly larger than $\mathrm{CMO}$,
and used it to obtain the compactness of commutators generated by a certain type of
bilinear Calder\'on--Zygmund operators
including smooth (inhomogeneous) bilinear Fourier multipliers
and bilinear pesudodifferential operators as special examples.
Indeed, $\mathrm{XMO}$ is strictly smaller than $\mathrm{VMO}$,
which was proved in \cite{txyyJFAA} by characterizing $\mathrm{XMO}$
in terms of the vanishing behavior of mean oscillations.

As a natural generalization of $\mathrm{BMO}$,
the Campanato space
$\mathcal{L}_\alpha$ defined in Definition \ref{defBVCXMO}
has also attracted a lot of attention after the advent of
 the celebrated work of Campanato \cite{c64}.
It is realized as the dual space of the real Hardy space $H^p$
with $p\in(0,1)$.
Recently, there exist some further studies \cite{ghwy18,ln24,tyyMANN}
on the vanishing subspace of the H\"older--Zygmund space $\Lambda_\alpha$ with $\alpha\in(0,1)$
(also called Lipschitz space in some context)
which is a special case of Campanato spaces; see Remark \ref{rem-1}(ii).
We now recall the definitions of these Campanato-type spaces.
In what follows, for any $s\in\mathbb{Z}_+$ (the set of all non-negative integers),
$P_Q^{s}(f)$  denotes the unique (minimal) polynomial of degree not greater than $s$ such that,
for any  $\gamma:=(\gamma_1,\ldots,\gamma_n)\in \mathbb{Z}_+^n:=(\mathbb{Z}_+)^n$ with
$|\gamma|:=\gamma_1+\cdots+\gamma_n$,
\begin{align}\label{PQsf}
\int_Q \left[f(x)-P_Q^{s}(f)(x)\right]x^\gamma\,dx=0
\end{align}
if $|\gamma|\le s$,
where $x^\gamma:=x_1^{\gamma_1}\cdots x_n^{\gamma_n}$ for any $x:=(x_1,\ldots,x_n)\in\mathbb{R}^n$.
A direct calculation shows
that $$P_Q^{0}(f)=\langle f\rangle_Q:=\fint_{Q}f(x)\,dx:=\frac{1}{|Q|}\int_{Q}f(x)\,dx.$$
It is well known that, for any $s\in\mathbb{Z}_+$,
there exists a constant $C_{(s)}\in[1,\infty)$, independent of $f$ and $Q$,
such that, for any $x\in Q$,
\begin{align}\label{Cs}
\left|P_Q^{s}(f)(x)\right|\le C_{(s)} \fint_Q |f(y)|\,dy.
\end{align}
Let $z \in \mathbb{R}^n$ and $Q$ be a cube.
The symbol $Q + z$ denotes the cube translated by the vector $z$.
Denote by $\ell(Q)$
the edge length of the cube $Q$.
Let $\alpha \in [0, \infty)$ and define
\begin{align*}
\lfloor \alpha \rfloor := \max \{ s \in \mathbb{Z}_+ :\ s \le \alpha \}
\text{ and }
\lceil \alpha \rceil := \min \{ s \in \mathbb{Z}_+ :\ s \ge \alpha \}.
\end{align*}

Here is the precise definition of Campanato spaces
and their closed subspaces.
In what follows, the limit $\lim_{a \to 0^+}$ means that
there exists $c_0\in(0,\infty)$ such that $a\in(0,c_0)$ and $a\to0$.
\begin{definition}\label{defBVCXMO}
Let $\alpha\in[0,\infty)$.
\begin{itemize}
\item[\rm (i)]
The \emph{Campanato space} $\mathcal{L}_\alpha$ is defined to be
the set of all locally integrable functions $f$ on $\mathbb{R}^n$ such that
\begin{equation}\label{bmo}
\|f\|_{\mathcal{L}_\alpha} :=
\sup_{\substack{\text{cubes } Q }}
\mathcal{O}_\alpha(f; Q) :=
\sup_{\substack{\text{cubes } Q }}
\frac{1}{|Q|^{1 + \frac{\alpha}{n}}} \int_Q \left| f(x) - P_Q^{\lfloor \alpha \rfloor}(f)(x) \right|\, dx
\end{equation}
is finite, where $P_Q^{\lfloor \alpha \rfloor}(f)$ denotes
the polynomial of degree at most $\lfloor \alpha \rfloor$
that minimizes the $L^1$-norm of the difference $f - P$ on $Q$ (as in~\eqref{PQsf}
with $s$ replaced by $\lfloor \alpha \rfloor$).
\item[\rm (ii)]
The \emph{vanishing Campanato space} ${\mathrm{V}}\mathcal{L}_\alpha$
is defined to be the subspace of $\mathcal{L}_\alpha$
consisting of all functions $f$ such that
\[
\lim_{a \to 0^+} \sup_{\{Q:\,\ell(Q) \le a\}} \mathcal{O}_\alpha(f; Q) = 0.
\]

\item[\rm (iii)]
The \emph{vanishing Campanato space} ${\mathrm{X}}\mathcal{L}_\alpha$
is defined to be the set of all functions $f \in {\mathrm{V}}\mathcal{L}_\alpha$
such that, for any fixed cube $Q$,
$$\lim_{z \to \infty} \mathcal{O}_\alpha(f; Q + z) = 0.$$

\item[\rm (iv)]
The \emph{vanishing Campanato space} ${\mathrm{C}}\mathcal{L}_\alpha$
is defined to be the set of all functions $f \in {\mathrm{X}}\mathcal{L}_\alpha$ such that
\[
\lim_{a \to \infty} \sup_{\{Q:\,\ell(Q) \ge a\}} \mathcal{O}_\alpha(f; Q) = 0.
\]
\end{itemize}
\end{definition}

From Definition \ref{defBVCXMO}, we infer that the following assertions hold:
\begin{itemize}
\item[\rm (i)] For any function $f \in \mathcal{L}_\alpha$,
the quantity $\mathcal{O}_\alpha(f; Q)$
is uniformly bounded for all cubes $Q$.

\item[\rm (ii)] For any function $f \in {\mathrm{V}}\mathcal{L}_\alpha$,
the quantity $\mathcal{O}_\alpha(f; Q)$
vanishes uniformly for small cubes $Q$.

\item[\rm (iii)] For any function $f \in {\mathrm{X}}\mathcal{L}_\alpha$,
the quantity $\mathcal{O}_\alpha(f; Q)$
vanishes uniformly for small or far cubes $Q$.

\item[\rm (iv)] For any function $f \in {\mathrm{C}}\mathcal{L}_\alpha$,
the quantity $\mathcal{O}_\alpha(f; Q)$
vanishes uniformly for large, small or far cubes $Q$.
\end{itemize}
These assertions can be summarized as in the following table.
\begin{table}[htb]
\begin{center}
\begin{tabular}{|c|c|c|c|c|}
\hline   \diagbox{\textbf{Oscillation}}{\textbf{Space}} & $\,\ \mathcal{L}_{\alpha}\,\ $ &  ${\mathrm{V}}\mathcal{L}_{\alpha}$ & ${\mathrm{X}}\mathcal{L}_{\alpha}$ & ${\mathrm{C}}\mathcal{L}_{\alpha}$ \\
\hline  \makecell{uniformly bounded \\ for all cubes} & \color{blue}\raisebox{-0.5ex}{\Large\checkmark} &\color{blue} \raisebox{-0.5ex}{\Large\checkmark} & \color{blue}\raisebox{-0.5ex}{\Large\checkmark} & \color{blue}\raisebox{-0.5ex}{\Large\checkmark}   \\
\hline  \makecell{uniformly vanish \\ for small cubes} & \color{red}\raisebox{-0.5ex}{\scriptsize\XSolid} & \color{blue}\raisebox{-0.5ex}{\Large\checkmark} & \color{blue}\raisebox{-0.5ex}{\Large\checkmark} & \color{blue}\raisebox{-0.5ex}{\Large\checkmark}  \\
\hline   \makecell{vanish \\ for far cubes} &  \color{red}\raisebox{-0.5ex}{\scriptsize\XSolid} & \color{red}\raisebox{-0.5ex}{\scriptsize\XSolid} & \color{blue}\raisebox{-0.5ex}{\Large\checkmark} & \color{blue}\raisebox{-0.5ex}{\Large\checkmark}  \\
\hline   \makecell{uniformly vanish \\ for large cubes} & \color{red}\raisebox{-0.5ex}{\scriptsize\XSolid} & \color{red}\raisebox{-0.5ex}{\scriptsize\XSolid} & \color{red}\raisebox{-0.5ex}{\scriptsize\XSolid} & \color{blue}\raisebox{-0.5ex}{\Large\checkmark}  \\
\hline
\end{tabular}
\caption{(Vanishing) Campanato spaces\label{table:1}}
\end{center}
\end{table}

\begin{remark}\label{rem-1}
The relationships between the Campanato-type spaces in Definition \ref{defBVCXMO}
and the known function spaces are as follows.
\begin{itemize}
\item[\rm (i)]
When $\alpha = 0$, the space $\mathcal{L}_{\alpha}$ coincides with the classical space $\mathrm{BMO}$. Consequently, the subspaces ${\mathrm{V}}\mathcal{L}_{\alpha}$, ${\mathrm{X}}\mathcal{L}_{\alpha}$, and ${\mathrm{C}}\mathcal{L}_{\alpha}$ correspond, respectively, to
\begin{itemize}
    \item $\mathrm{VMO}$ as in \cite{s75},
    \item $\mathrm{XMO}$ as in \cite{txyyJFAA, tx20}, and
    \item $\mathrm{CMO}$ as in \cite{u78}.
\end{itemize}
In this case, we write $\mathcal{O}(f; Q)$ in place of $\mathcal{O}_0(f; Q)$ for simplicity.

\item[\rm (ii)]
If $\alpha\in(0,1)$, then the space $\mathcal{L}_{\alpha}$ coincides with
the H\"older--Zygmund space $\Lambda_{\alpha}$; see \cite{m64}.
In this case, we also denote $\mathcal{L}_{\alpha}$ by $\mathrm{BMO}_\alpha$.
Its subspaces then satisfy
\begin{itemize}
    \item ${\mathrm{X}}\mathcal{L}_{\alpha} = \mathrm{XMO}_\alpha$ as in \cite{tyyMANN} and
    \item ${\mathrm{C}}\mathcal{L}_{\alpha} = \mathrm{CMO}_\alpha$ as in \cite{ghwy18}.
\end{itemize}

\item[\rm (iii)]
The space $\mathcal{L}_{1}$ coincides with the Lipschitz space
${\rm Lip}$. To the best of our knowledge, the corresponding vanishing-type subspaces have not been systematically studied in the literature.

\item[\rm (iv)]
When $\alpha \in (1, \infty)$, a function $f$ belongs to $\mathcal{L}_{\alpha}$
if and only if the $(\lceil \alpha \rceil - 1)$th derivatives of $f$ belong to
the H\"older space $\Lambda_{\alpha - \lfloor \alpha \rfloor}$;
see, for instance, \cite{blyy21, c64} for more details.
\end{itemize}
\end{remark}

Recall that the well-known John--Nirenberg inequality illustrates
the self-improvement property of the space $\mathrm{BMO}$.
That is, one can replace the space $L^1$ in the definition \eqref{bmo}
with $\alpha = 0$ by the space $L^q$ for any $q \in (1, \infty)$.
This self-improvement phenomenon has been systematically studied by
Berkovits et al.~\cite{bkm16} by means of an abstract good-$\lambda$ inequality.

It is notable that a useful framework, called the ball Banach function
space $X$, is invented in \cite{shyy17}.
In this article, we investigate the corresponding self-improvement properties of vanishing Campanato spaces within a general framework
based on $X$.
In particular, our results provide new characterizations of vanishing Campanato spaces. The class $X$ includes not only classical Lebesgue spaces $L^q$ for $q \in [1, \infty)$, but also a rich collection of other function spaces such as weighted Lebesgue spaces, Morrey spaces, variable Lebesgue spaces, and mixed-norm Lebesgue spaces.

Recently, there has been growing interest in the study of ball Banach function spaces. For further developments and recent advances on this topic, we refer to \cite{dlyyz, dpgyyz, fsyy, syy24, tyyzPA, zhyy22, zwyy21, zlyyz24}, as well as the survey \cite{ln24}.

\begin{definition}\label{def-bbfs}
A quasi-Banach space $X\subset\mathscr M$ is called a
\emph{ball quasi-Banach function space} if it satisfies
\begin{itemize}
\item[(i)] $\|f\|_X=0$ implies that $f=0$ almost everywhere;
\item[(ii)] $|g|\le |f|$ almost everywhere implies that $\|g\|_X\le\|f\|_X$;
\item[(iii)] $0\le f_m\uparrow f$ almost everywhere implies that $\|f_m\|_X\uparrow\|f\|_X$
as $m\to\infty$;
\item[(iv)] for any ball $B$ in $\mathbb{R}^n$, $\mathbf{1}_B\in X$.
\end{itemize}
Moreover, a ball quasi-Banach function space $X$ is called a
\emph{ball Banach function space} if the norm of $X$
satisfies the triangle inequality: for any $f,\ g\in X$,
\begin{equation*}
\|f+g\|_X\le \|f\|_X+\|g\|_X,
\end{equation*}
and, for any ball $B$ of $\mathbb{R}^n$, there exists a positive constant $C_{(B)}$,
depending on $B$, such that, for any $f\in X$,
\begin{equation*}
\int_B|f(x)|\,dx\le C_{(B)}\|f\|_X.
\end{equation*}
\end{definition}

\begin{remark}
\begin{enumerate}[{\rm (i)}]
\item As mentioned in \cite[Remark 2.5(ii)]{YHYY22ams},
we obtain an equivalent formulation of
Definition \ref{def-bbfs} via replacing any ball
$B$ in $\mathbb{R}^n$ by any bounded measurable set
$S$ in $\mathbb{R}^n$.

\item In Definition \ref{def-bbfs}, if we replace any ball $B$
by any measurable $E$ with $|E|<\infty$, then
we obtain the definition of \emph{Banach function spaces},
which was originally
introduced by Bennett and Sharpley in
\cite[Chapter 1, Definitions 1.1 and 1.3]{bs}.
Using their definitions,
we easily find that
a Banach function space is always a ball Banach function space.
However,
the converse is not necessary to be true
(see, for instance, \cite[p.\,9]{shyy17}).

\item In Definition \ref{def-bbfs}, if we replace (iv)
by the following \emph{saturation property}:
\begin{enumerate}[{\rm (a)}]
\item[]
for any measurable set
$E$ in $\mathbb{R}^n$
with $|E|\in (0,\infty)$, there exists a measurable
set $F\subset E$ with $|F|\in (0,\infty)$ satisfying that
${\bf 1}_{F}\in X $,
\end{enumerate}
then we obtain the definition of Banach function spaces in the terminology
of Lorist and Nieraeth \cite[p.\,251]{ln24}. Moreover,
by \cite[Proposition 2.5]{zyy24}
(see also \cite[Proposition 4.21]{n23}),
we conclude that, if the normed vector space $X $ under
consideration satisfies the additional assumption that
the Hardy--Littlewood maximal operator $M$ is weakly bounded on one of its convexification,
then the definition of Banach function spaces in \cite{ln24}
coincides with the definition of ball
Banach function spaces. Thus, under this additional assumption, working with
ball Banach function spaces in the sense of Definition \ref{def-bbfs} or Banach
function spaces in the sense of \cite{ln24} would
yield exactly the same results.

\item  From \cite[Proposition 1.2.36]{lyh22}, we deduce that
both (ii) and (iii) of Definition \ref{def-bbfs}
imply that any ball Banach function space is complete.
\end{enumerate}
\end{remark}

Here is the definition of $X$-based Campanato spaces,
which we focus on in this paper.
\begin{definition}\label{def-oscX}
Let $\alpha \in [0,\infty)$, $X$ be a ball quasi-Banach function space,
$f \in L^1_{\mathrm{loc}}$, and $Q$ be a cube in $\mathbb{R}^n$.
The \emph{$X$-based Campanato seminorm $\mathcal{O}_{\alpha,X}(f;Q)$}
is defined by setting
\[
\mathcal{O}_{\alpha,X}(f;Q) := \frac{|Q|^{-\frac{\alpha}{n}}}{\|\mathbf{1}_Q\|_X} \left\| \left[f - P_Q^{\lfloor\alpha\rfloor}(f)\right]\mathbf{1}_Q \right\|_X,
\]
where $P_Q^{\lfloor\alpha\rfloor}(f)$ is the minimal polynomial as in \eqref{PQsf}
with $s$ replaced by $\lfloor\alpha\rfloor$.

Moreover, the \emph{spaces $\mathcal{L}_{\alpha,X}$, $\mathrm{V}\mathcal{L}_{\alpha,X}$,
$\mathrm{X}\mathcal{L}_{\alpha,X}$, and $\mathrm{C}\mathcal{L}_{\alpha,X}$}
are defined analogously to the corresponding spaces in
Definition~\ref{defBVCXMO} (and Table~\ref{table:1})
with $\mathcal{O}_{\alpha}(f;Q)$ replaced by $\mathcal{O}_{\alpha,X}(f;Q)$.
\end{definition}

Let $r\in(0,\infty)$ and $f \in L^1_{\mathrm{loc}}$.
Define the \emph{ball average}
\[
B_r(f)(x) := \frac{1}{|B(x,r)|} \int_{B(x,r)} |f(y)|\,dy.
\]
for any $x \in \mathbb{R}^n$.
Moreover, the \emph{Hardy--Littlewood maximal operator $M$} is defined by setting
\begin{align}\label{Mf}
M(f)(x) := \sup_{r\in(0,\infty)} B_r(f)(x),
\end{align}
for any locally integrable function $f$ and any $x \in \mathbb{R}^n$.

If $\alpha \in [0,1)$, then $P_Q^{\lfloor\alpha\rfloor}(f) = \langle f \rangle_Q$,
as previously noted. In this case, the space $\mathcal{L}_{\alpha,X}$ was
introduced in \cite{is17}. Let $X$ be a Banach function space such that
the Hardy--Littlewood maximal operator
$M$ is bounded on the associated K\"othe dual space $X'$
(See Definition~\ref{def-X'} for its precise definition).
It then follows from \cite[Theorems 1.1 and 1.3]{is17} that
\begin{align}\label{LX=L}
\mathcal{L}_{\alpha,X} = \mathcal{L}_{\alpha}
\end{align}
with equivalent norms.

Given this setting, a natural \emph{question} arises: does the equivalence
\eqref{LX=L}
still hold for the corresponding vanishing subspaces?

We now present the main result of this article, which provides an affirmative answer to the above question. Furthermore, it extends the admissible range for the parameter $\alpha$ and characterizes the vanishing Campanato spaces in terms of the ball Banach function space $X$.

\begin{theorem}\label{thm1}
Let $\alpha \in [0, \infty)$ and let $X$ be a ball Banach function space such that the Hardy--Littlewood maximal operator $M$ is bounded on the associate space $X'$. Then
$\mathrm{Y}\mathcal{L}_{\alpha} = \mathrm{Y}\mathcal{L}_{\alpha, X}$
for any $\mathrm{Y} \in \{\mathrm{V}, \mathrm{X}, \mathrm{C}\}.$
\end{theorem}

The proof of Theorem \ref{thm1} is given in Section \ref{s4}.
One may speculate that Theorem \ref{thm1} is an immediately consequence of \eqref{LX=L}.
However, that is not the case.
Indeed, the proof of \eqref{LX=L} is based not  on
$\mathcal{O}_{\alpha}(f;Q)\sim\mathcal{O}_{\alpha,X}(f;Q)$
but on a combination of both
$\mathcal{O}_{\alpha}(f;Q)\lesssim\mathcal{O}_{\alpha,X}(f;Q)
\lesssim\mathcal{O}_{\alpha,L^q}(f;Q)$ for some $q\in(1,\infty)$
and the self-improving property of Campanato spaces.
In this article, for the oscillation $\mathcal{O}_{\alpha,X}(f;Q)$,
we still use the above lower bound estimates
and establish some upper bound estimates in terms of (higher-order) differences;
see Lemmas \ref{lem-oscX} and \ref{lem-2diff}.
\begin{remark}\label{rem-2}
To the best of our knowledge, Theorem~\ref{thm1} is entirely new even in the special case $\alpha = 0$.
\begin{itemize}
\item[\rm (i)] The case $\alpha = 0$ in Theorem~\ref{thm1} can be proved
by using classical characterizations of vanishing BMO spaces,
specifically Lemma~\ref{lem-vcxmo}.

\item[\rm (ii)] The case $\alpha \in (0,1)$ is handled by establishing new characterizations of
$\mathrm{VMO}_\alpha$, $\mathrm{XMO}_\alpha$, and $\mathrm{CMO}_\alpha$ in terms of the pointwise difference $|f(x)-f(y)|$; see Proposition~\ref{prop-lip}.

\item[\rm (iii)] For $\alpha = 1$, we prove Theorem~\ref{thm1} by developing new characterizations of vanishing Campanato spaces using the second-order difference $|f(x+y)+f(x-y)-2f(x)|$; see Proposition~\ref{prop-zyg}.

\item[\rm (iv)] The argument for any $\alpha \in(1,\infty)$ is essentially similar to that
for any $\alpha \in (0,1]$, as noted in Remark~\ref{rem-1}(iv).
Indeed, the proof for the case $\alpha = 1$ can be adapted to cover the case $\alpha \in (1, \infty)$;
see, for example, \cite[pp.\,300-302, Lemma~5.18 and Theorem~5.22]{gr85}.
Therefore, we restrict our proof of Theorem~\ref{thm1} to the range $\alpha \in [0,1]$
and omit the case $\alpha \in(1,\infty)$.

\item[\rm (v)] If $\alpha<0$, then \eqref{LX=L} fails.
To show this, choose $X$ to be the Morrey space and see Remark \ref{rem-morrey}(ii). In this sense, the range $\alpha$
in Theorem \ref{thm1} is \emph{sharp}.
\end{itemize}
\end{remark}
Theorem \ref{thm1} reveals a more general self-improvement
phenomenon of vanishing Campanato spaces. In \eqref{bmo},
one can not only improve the integrability of
$f - P_Q^{\lfloor\alpha\rfloor}(f)$ from $q = 1$ to $q \in (1,\infty)$, but
also add a Muckenhoupt weight $w \in A_q$ in the integral
because $X$ can be chosen as the weighted space $L^q_w$
and $M$ is bounded on the associated space of $L^q_w$.
A key novelty of this article lies in using higher-order differences
to dominate oscillations.
Rather than directly estimating higher-order differences, the domination is achieved by
smoothing the function through the tool of convolution.
The earliest formulation of this convolution method can be traced back to
Garc\'ia-Cuerva and Rubio de Francia \cite{gr85}.
We modify their convolution technique to study the vanishing behaviors of
higher-order oscillations.
As byproducts, we also establish some new characterizations
of vanishing Campanato spaces in terms of (higher-order) differences;
see Propositions \ref{prop-lip} and \ref{prop-zyg}.

We next describe the structure of the remainder of this article.
Section \ref{s2} collects some elementary estimates on oscillations.
Section \ref{s3} characterizes vanishing Campanato spaces in terms of (higher-order) differences.
Section \ref{s4} proves Theorem \ref{thm1} in Section \ref{s1} and apply Theorem \ref{thm1}
to some specific function spaces.
Finally, in Appendix \ref{a}, we present several examples to show that
these vanishing Campanato spaces naturally arise in the study on
the compactness of fractional integral commutators in Morrey spaces.

We end this introduction by making some conventions on symbols.
Throughout this article,
let ${\mathbb Z}$ be the collection of all integers and ${\mathbb Z}_+:=\{0,1,\dots\}$.
We always use $C$ to denote a positive constant independent of the main parameters involved.
The symbol $f\lesssim g$ means $f\leq  Cg$
and, if $f\lesssim g\lesssim f$, we then write $f\sim g$.
Denote $\mathscr M$ by the set of all measurable functions on $\mathbb{R}^n$.
For any $p\in [1,\infty]$, let $p'$ be the number that satisfies $\frac1p+\frac1{p'}=1$.
Moreover, for any ball Banach function space $X$, let $X'$ be its associate space
(also called the K\"othe dual);
see Definitions \ref{def-bbfs} and \ref{def-X'} for their definitions.
For any function $g$ and any $x,\,h\in\mathbb{R}^n$,
its second order difference $\Delta^2_h g$
is defined by setting $\Delta^2_h g(x):=g(x+h)+g(x-h)-2g(x)$.
For any function $\phi$ and any $r\in(0,\infty)$, let $\phi_r(\cdot):=r^{-n}\phi(\frac\cdot{r})$.
For any $x\in\mathbb{R}^n$ and $r\in(0,\infty)$,
the ball $B(x,r):=\{y\in\mathbb{R}^n:\ |y-x|<r\}$.
The limit $\lim_{a \to 0^+}$ means that there exists $c_0\in(0,\infty)$
such that $a\in(0,c_0)$ and $a\to0$.
Finally, in all proofs, we consistently retain the symbols
introduced in the original theorem (or related statement).

\section{Oscillations via Ball Banach Function Spaces}\label{s2}

\begin{definition}\label{def-X'}
For any ball quasi-Banach function space $X$, the \emph{associate space} (also called the
\emph{K\"othe dual}) $X'$ of $X$ is defined by setting
\begin{equation*}
X':=\left\{f\in\mathscr M:\ \|f\|_{X'}:=
\sup_{{g\in X,\,\|g\|_X=1}}\|fg\|_{L^1}<\infty\right\},
\end{equation*}
where $\|\cdot\|_{X'}$ is called the \emph{associate norm} of $\|\cdot\|_X$.
\end{definition}

The following Lorentz--Luxemburg theorem can be found in the book
of Bennett and Sharpley \cite{bs}; see also \cite[Lemma 2.1]{is17}.
\begin{lemma}
Let $X$ be a ball Banach function space.
Then $X=(X')'$ hold and, in particular, the norms $\|\cdot\|_{X}$
and $\|\cdot\|_{(X')'}$ are equivalent.
\end{lemma}

Using this and Definitions \ref{def-bbfs} and \ref{def-X'},
we immediately obtain the following H\"older inequality of $X$.
\begin{lemma}\label{holder}
Let $X$ be a ball quasi-Banach function space with the associate space $X'$.
If $f\in X$ and $g\in X'$, then $fg$ is integrable and
\begin{equation*}
\int_{\mathbb{R}^n}|f(x)g(x)|\,dx\le \|f\|_X\|g\|_{X'}.
\end{equation*}
\end{lemma}

By \cite[Lemma 2.2 and Remark 2.3]{is17},
we have the partial converse of Lemma \ref{holder}.
\begin{lemma}\label{reholder}
Let $X$ be a ball quasi-Banach function space such that the Hardy--Littlewood maximal operator $M$ is bounded on
its associated space $X'$.
Then there exists some positive constant $C$ such that,
for any cube $Q$ in $\mathbb{R}^n$,
$\|\mathbf{1}_Q\|_X\|\mathbf{1}_Q\|_{X'}
\le C |Q|.$
\end{lemma}

The following result shows that, under the assumption that $M$ is bounded on $X'$,
the non-zero constant function does not belong to $X$.
\begin{lemma}\label{const}
Let $X$ be a ball quasi-Banach function space such that
the Hardy--Littlewood maximal operator $M$ is bounded on
its associated space $X'$.
Then $\mathbf{1}_{\mathbb{R}^n}\notin X$.
\end{lemma}
\begin{proof}
By Definition \ref{def-bbfs}(iv),
the boundedness of $M$ on $X'$, Lemma \ref{holder}, and
the fact $M(\mathbf{1}_{B(\mathbf{0},1)})\notin L^1$,
we obtain
\begin{align*}
\|\mathbf{1}_{\mathbb{R}^n}\|_X \|M(\mathbf{1}_{B(\mathbf{0},1)})\|_{X'}
\ge\int_{\mathbb{R}^n} \mathbf{1}_{\mathbb{R}^n}(x)M(\mathbf{1}_{B(\mathbf{0},1)})(x)\,dx
\gtrsim\int_{\mathbb{R}^n} \frac{1}{(1+|x|)^n}\,dx=\infty,
\end{align*}
proving $\mathbf{1}_{\mathbb{R}^n}\notin X$.
\end{proof}

We now establish a connection between the oscillation and the pointwise difference.

\begin{lemma}\label{lem-oscX}
Let $\alpha \in [0,1)$ and $X$ be a ball Banach function space.
Then there exists a positive constant $C_1$, depending only on $n$,
such that,
for any cube $Q$ in $\mathbb{R}^n$ and any $f \in L^1_{\mathrm{loc}}$,
the following inequality holds:
\begin{equation}\label{eq:250419-1}
\mathcal{O}_{\alpha, X}(f; Q) \le C_1 \sup_{{x, y \in Q,\, x \neq y}} \frac{|f(x) - f(y)|}{|x - y|^\alpha}.
\end{equation}

If the Hardy--Littlewood maximal operator $M$ is bounded
on the associate space $X'$, then there exists a positive constant $C_2$,
depending only on $n$ and $X$, such that,
for any cube $Q$ in $\mathbb{R}^n$ and any $f\in L^1_{\mathrm{loc}}$,
\begin{equation}\label{eq:250419-2}
\mathcal{O}_{\alpha}(f; Q)
:= |Q|^{-\frac{\alpha}{n}} \fint_Q |f(x) - \langle f \rangle_Q|\,dx
\le C_2 \mathcal{O}_{\alpha, X}(f; Q).
\end{equation}
\end{lemma}

\begin{proof}
By the definition of $\mathcal{O}_{\alpha, X}(f; Q)$, we have
\[
\mathcal{O}_{\alpha, X}(f; Q)
= \frac{|Q|^{-\frac{\alpha}{n}}}{\|\mathbf{1}_Q\|_X} \left\|\left(f - \langle f \rangle_Q\right)\mathbf{1}_Q\right\|_X
= \frac{|Q|^{-\frac{\alpha}{n}}}{\|\mathbf{1}_Q\|_X} \left\|\fint_Q [f(\cdot) - f(y)]\,dy \, \mathbf{1}_Q\right\|_X.
\]
Using the deifition of the Lipschitz norm, we further obtain
\begin{align*}
\mathcal{O}_{\alpha, X}(f; Q)
&\le \frac{|Q|^{-\frac{\alpha}{n}}}{\|\mathbf{1}_Q\|_X} \left\|\sup_{x,\,y \in Q} |f(x) - f(y)| \cdot \mathbf{1}_Q\right\|_X\\
&= |Q|^{-\frac{\alpha}{n}} \sup_{x,\,y \in Q} |f(x) - f(y)|
\lesssim \sup_{\genfrac{}{}{0pt}{}{x,\,y \in Q}{x \neq y}} \frac{|f(x) - f(y)|}{|x - y|^\alpha},
\end{align*}
which proves \eqref{eq:250419-1}.

To prove \eqref{eq:250419-2}, assume that $M$ is bounded on $X'$. Then, from the H\"older inequality for Banach function spaces and its reverse (cf. Lemmas~\ref{holder} and~\ref{reholder}), it follows that
\begin{align*}
|Q|^{-\frac{\alpha}{n}} \fint_Q |f(x) - \langle f \rangle_Q|\,dx
&\le |Q|^{-\frac{\alpha}{n}} \frac{1}{|Q|} \|\mathbf{1}_Q\|_{X'} \left\|\left(f - \langle f \rangle_Q\right) \mathbf{1}_Q\right\|_X \\
&\lesssim |Q|^{-\frac{\alpha}{n}} \frac{1}{\|\mathbf{1}_Q\|_X} \left\|\left(f - \langle f \rangle_Q\right) \mathbf{1}_Q\right\|_X
= \mathcal{O}_{\alpha, X}(f; Q),
\end{align*}
which completes the proof of \eqref{eq:250419-2}
and hence Lemma \ref{lem-oscX}.
\end{proof}
Next, we present the higher-order version of Lemma~\ref{lem-oscX}.
In what follows, for any function $g$ and any $x, h \in \mathbb{R}^n$,
the \emph{second-order difference} $\Delta^2_h g(x)$ is defined by setting
\[
\Delta^2_h g(x) := g(x + h) + g(x - h) - 2g(x).
\]

\begin{lemma}\label{lem-2diff}
Let $X$ be a ball Banach function space, $f \in L^1_{\mathrm{loc}}$,
and $Q$ be a cube in $\mathbb{R}^n$ centered at $x_0$ with edge length $2r\in(0,\infty)$.
Then the following statements hold.
\begin{itemize}
\item[$({\rm i})$] There exists a polynomial $P_1$ of degree $1$ and a positive constant $C_1$, depending only on $n$, such that, for any $x \in Q$,
\[
|f(x) - P_1(x)| \le C_1 r \sup_{0 < |y| \leq r} \frac{|\Delta^2_y f(x)|}{|y|}.
\]

\item[$({\rm ii})$] There exists a positive constant $C_2$, depending only on $n$, such that
\[
\mathcal{O}_{1, X}(f; Q) \leq C_2 \sup_{{x \in Q},\,{0 < |y| \leq r}} \frac{|\Delta^2_y f(x)|}{|y|}.
\]

\item[$({\rm iii})$] If the Hardy--Littlewood maximal operator $M$ is bounded on the associate space $X'$, then there exists a positive constant $C_3$,
depending only on $n$ and $X$, such that
\[
\mathcal{O}_1(f; Q) := |Q|^{-1/n} \fint_Q \left| f(x) - P^1_Q(f)(x) \right| \, dx
\leq C_3 \mathcal{O}_{1, X}(f; Q).
\]
\end{itemize}
\end{lemma}

\begin{proof}
We begin with showing (i).
Let $\phi \in C^\infty$ be an even function supported in $B(\mathbf{0}, 1)$
with $\int_{\mathbb{R}^n} \phi(x) \, dx = 1$.
For any given $r\in(0,\infty)$, define $\phi_r(x) := r^{-n} \phi(x/r)$ for any $x\in\mathbb{R}^n$.
Then, for any multi-index $\gamma \in \mathbb{Z}_+^n$ with $|\gamma| = 2$,
we have $D^\gamma(\phi_r) = r^{-2}(D^\gamma \phi)_r$, which is even and has integral zero.
A direct consequence of this fact is that, for any $x\in\mathbb{R}^n$,
\begin{align*}
2 D^\gamma \psi(x)
&=2 D^\gamma(f * \phi_r)(x)
=2 f * D^\gamma(\phi_r)(x)\\
&= \int_{\mathbb{R}^n} D^\gamma(\phi_r)(y) [f(x - y) + f(x + y) - 2f(x)] \, dy .
\end{align*}

Let $Q$ be a cube centered at $x_0$ with edge length $2r$, and define $\psi := f * \phi_r$.
Then, for any $x \in \mathbb{R}^n$,
\begin{align}\label{Dpsi}
|2 D^\gamma \psi(x)|
\leq \int_{B(\mathbf{0}, r)} \frac{|y|}{r^2} |(D^\gamma \phi)_r(y)| \sup_{0 < |z| \leq r} \frac{|\Delta^2_z f(x)|}{|z|} \, dy
\sim r^{-1} \sup_{0 < |y| \leq r} \frac{|\Delta^2_y f(x)|}{|y|}.
\end{align}

Let $P_1$ be the Taylor polynomial of $\psi$ at $x_0$ of degree $1$.
By the Taylor theorem and \eqref{Dpsi}, for any $x \in Q$,
\begin{equation}\label{psi-P1}
|\psi(x) - P_1(x)| \lesssim r^{-1} \sup_{0 < |y| \leq r} \frac{|\Delta^2_y f(x)|}{|y|} |x - x_0|^2
\lesssim r \sup_{0 < |y| \leq r} \frac{|\Delta^2_y f(x)|}{|y|}.
\end{equation}

Since $\phi_r$ is even, we also obtain, for any $x \in \mathbb{R}^n$,
\begin{align*}
f(x) - \psi(x)
&=f(x) - \int_{\mathbb{R}^n} \phi_r(y) f(x - y) \, dy\\
&= -\frac{1}{2} \int_{\mathbb{R}^n} \phi_r(y) [f(x - y) + f(x + y) - 2f(x)] \, dy
\end{align*}
and hence
\begin{align*}
|f(x) - \psi(x)|
\lesssim \int_{\mathbb{R}^n} |\phi_r(y)| \sup_{0 < |z| \leq r} \frac{|\Delta^2_z f(x)|}{|z|} |y| \, dy
\lesssim r \sup_{0 < |y| \leq r} \frac{|\Delta^2_y f(x)|}{|y|}.
\end{align*}

Combining this with \eqref{psi-P1}, we find that, for any $x \in Q$,
\[
|f(x) - P_1(x)| \lesssim r \sup_{0 < |y| \leq r} \frac{|\Delta^2_y f(x)|}{|y|},
\]
which proves (i).

To prove (ii), let $P_1$ be as in (i). Then
\begin{align*}
\left\| [f - P^1_Q(f)] \mathbf{1}_Q \right\|_X
&\leq \left\| [f - P_1] \mathbf{1}_Q \right\|_X + \left\| [P^1_Q(P_1 - f)] \mathbf{1}_Q \right\|_X \\
&\lesssim |Q|^{1/n} \sup_{\genfrac{}{}{0pt}{}{x \in Q}{0 < |y| \leq r}} \frac{|\Delta^2_y f(x)|}{|y|} \|\mathbf{1}_Q\|_X
+ \left\| \fint_Q |f(x) - P_1(x)| \, dx \, \mathbf{1}_Q \right\|_X \\
&\lesssim |Q|^{1/n} \sup_{\genfrac{}{}{0pt}{}{x \in Q}{0 < |y| \leq r}} \frac{|\Delta^2_y f(x)|}{|y|} \|\mathbf{1}_Q\|_X.
\end{align*}
Therefore,
\[
\mathcal{O}_{1, X}(f; Q) = \frac{|Q|^{-1/n}}{\|\mathbf{1}_Q\|_X} \left\| [f - P^1_Q(f)] \mathbf{1}_Q \right\|_X
\lesssim \sup_{\genfrac{}{}{0pt}{}{x \in Q}{0 < |y| \leq r}} \frac{|\Delta^2_y f(x)|}{|y|},
\]
which completes the proof of (ii).

The proof of (iii) follows the same reasoning as in Lemma~\ref{lem-oscX}, and is thus omitted.
\end{proof}

\section{Characterizations of Vanishing Campanato Spaces}\label{s3}

The equivalence \eqref{LX=L} with $\alpha = 0$ was first established by Ho~\cite{h12}
as a byproduct of the atomic decomposition via Banach function spaces.
Later, Izuki~\cite[Theorem 3.1]{i16} provided the other simple proof
by employing the Rubio de Francia algorithm; see also~\cite{cm18, cw17, c16, ins19, n23}
for further applications of this powerful technique.

The following characterization for any $\alpha \in [0,1)$ is stated in \eqref{LX=L} and,
as pointed out in~\cite{is17}, the remaining case $\alpha \in [1,\infty)$ follows from
\cite[Theorem 1.2]{is17} and \cite[p.~292]{gr85}. Therefore, we omit the details here.

\begin{proposition}\label{bmox}
Let $\alpha \in [0, \infty)$ and $X$ be a ball Banach function space such that the Hardy--Littlewood maximal operator $M$ is bounded on its associate space $X'$. Then
$\mathcal{L}_{\alpha,X} = \mathcal{L}_{\alpha}$
in the sense of equivalent norms.
\end{proposition}

The following lemma summarizes key results from~\cite{s75, txyyJFAA, u78}.

\begin{lemma}\label{lem-vcxmo}
\begin{itemize}
  \item[\rm (i)] Let $C_{\mathrm{u}}$ denote the set of all uniformly continuous functions on $\mathbb{R}^n$. Then
  \[
    \mathrm{VMO} = \overline{C_{\mathrm{u}} \cap \mathrm{BMO}}^{\mathrm{BMO}}.
  \]

  \item[\rm (ii)] Let $B_{\infty}$ be the set of all infinitely differentiable functions on $\mathbb{R}^n$ whose derivatives of all orders vanish at infinity. Then
  $\mathrm{XMO} = \overline{B_{\infty} \cap \mathrm{BMO}}^{\mathrm{BMO}}.$

  \item[\rm (iii)] Let $C_{\mathrm{c}}^\infty$ denote the set of all compactly supported smooth functions. Then
  $\mathrm{CMO} = \overline{C_{\mathrm{c}}^\infty}^{\mathrm{BMO}}.$
\end{itemize}
\end{lemma}

It is well known, using approximation by means of convolution, that $\mathrm{CMO}$ also coincides with the $\mathrm{BMO}$-closure of
\[
  C_0 := \left\{ f \in C : \lim_{|x| \to \infty} f(x) = 0 \right\}.
\]
Moreover, by~\cite[Theorem 1.2]{txyyJFAA}, $\mathrm{XMO}$ also coincides with the $\mathrm{BMO}$-closure of
\[
  B_1 := \left\{ f \in C^1 : \lim_{|x| \to \infty} |\nabla f(x)| = 0 \right\}.
\]
It is clear that $B_1 \subsetneqq C_0$, which highlights a typical difference between $\mathrm{CMO}$ and $\mathrm{XMO}$.

\begin{lemma}\label{lem:250509-1}
Let $a\in(0,\infty)$, $\alpha\in(0,1),$ and
$f \in L^1_{\rm loc}$. Then
\[
\sup_{\genfrac{}{}{0pt}{}{x,\,y\in\mathbb{R}^n}{0<|x-y|\le a}}
\frac{|f(x)-f(y)|}{|x-y|^\alpha}
\le \sup_{\ell(Q)\le2a} \mathcal{O}_\alpha(f;Q).
\]
\end{lemma}

\begin{proof}
Take two different Lebesgue points $x$, $y\in\mathbb{R}^n$ of $f$. Choose a half-open cube $Q_0$ containing $x$ and $y$ with edge length $\ell(Q_0)$ satisfying $|x-y|<\ell(Q_0)<2|x-y|$. Consider the dyadic descendants of $Q_0$.
For any $k\in\mathbb{N}$, define $Q_k(x)$ to be the dyadic descendants of $Q_0$ containing $x$ with edge length $2^{-k}\ell(Q_0)$. We similarly define $Q_k(y)$ for any $k\in\mathbb{N}$. Then, by the Lebesgue differentiation theorem, we conclude that, for any Lebesgue point $x$ of $f$,
\[
f(x)=\lim_{k\to\infty}\langle f\rangle_{Q_{k}(x)}
=\sum_{k=0}^\infty \left[\langle f\rangle_{Q_{k+1}(x)}
-\langle f\rangle_{Q_{k}(x)}\right] +\langle f\rangle_{Q_{0}}.
\]
This implies that
\begin{align}\label{fx-fy}
f(x)-f(y)=\sum_{k=0}^\infty \left[\langle f\rangle_{Q_{k+1}(x)}
-\langle f\rangle_{Q_{k}(x)}\right]
-\sum_{k=0}^\infty \left[\langle f\rangle_{Q_{k+1}(y)}
-\langle f\rangle_{Q_{k}(y)}\right].
\end{align}
We estimate the terms individually. First, observe that
\[
\left| \sum_{k = 0}^\infty \left[ \langle f \rangle_{Q_{k+1}(x)} - \langle f \rangle_{Q_k(x)} \right] \right|
\le \sum_{k = 0}^\infty \left| \langle f \rangle_{Q_{k+1}(x)} - \langle f \rangle_{Q_k(x)} \right|
\]
by the triangle inequality.
For each term in the summation,
\[
\left| \langle f \rangle_{Q_{k+1}(x)} - \langle f \rangle_{Q_k(x)} \right|
\le  \fint_{Q_{k+1}(x)} \left| f(y) - \langle f \rangle_{Q_{k}(x)}\right| dy
\le 2^n\fint_{Q_{k}(x)} \left| f(y) - \langle f \rangle_{Q_{k}(x)}\right| dy.
\]
Using the oscillation estimate,
\begin{align*}
\fint_{Q_k(x)} |f(z) - \langle f \rangle_{Q_k(x)}| \, dz
&\le |Q_k(x)|^{\frac\alpha n} \sup_{Q \subset Q_0} \mathcal{O}_\alpha(f; Q)\\
&= \left[ 2^{-k} \ell(Q_0) \right]^\alpha \sup_{Q \subset Q_0} \mathcal{O}_\alpha(f; Q).
\end{align*}
Summing up,
\[
\left| \sum_{k = 0}^\infty \left[ \langle f \rangle_{Q_{k+1}(x)} - \langle f \rangle_{Q_k(x)} \right] \right|
\lesssim [\ell(Q_0)]^\alpha \sup_{Q \subset Q_0} \mathcal{O}_\alpha(f; Q).
\]
Since $|x - y| < \ell(Q_0) < 2|x - y|$, we conclude that
\[
\left| \sum_{k = 0}^\infty \left[ \langle f \rangle_{Q_{k+1}(x)} - \langle f \rangle_{Q_k(x)} \right] \right|
\lesssim |x - y|^\alpha \sup_{Q \subset Q_0} \mathcal{O}_\alpha(f; Q).
\]
A similar estimate holds for the series centered at $y$, and hence
\[
|f(x)-f(y)|\lesssim |x-y|^\alpha \sup_{Q\subset Q_0} \mathcal{O}_\alpha(f;Q).
\]
Therefore,
\[
\sup_{\genfrac{}{}{0pt}{}{x,\,y\in\mathbb{R}^n}{0<|x-y|\le a}}\frac{|f(x)-f(y)|}{|x-y|^\alpha}
\lesssim \sup_{\genfrac{}{}{0pt}{}{Q\subset Q_0}{\ell(Q_0)\le 2a}}\mathcal{O}_\alpha(f;Q)
\le \sup_{\ell(Q)\le2a} \mathcal{O}_\alpha(f;Q),
\]
which completes the proof of Lemma \ref{lem:250509-1}.
\end{proof}
Lemmas~\ref{lem-oscX} and~\ref{lem-vcxmo} are two key ingredients in the proof of Theorem~\ref{thm1} for the case $\alpha = 0$. When $\alpha \in (0,1)$, in order to apply Lemma~\ref{lem-oscX}, we need to characterize the vanishing subspaces in terms of the quantity
$\frac{|f(x) - f(y)|}{|x - y|^\alpha}.$
Note that, for any $\alpha \in [0,1]$, we have
\begin{align}\label{osc-lip}
|Q|^{-\frac{\alpha}{n}} \fint_Q |f(x) - \langle f \rangle_Q|\,dx
&\le |Q|^{-\frac{\alpha}{n}} \fint_Q \fint_Q |f(x) - f(y)|\,dx\,dy  \notag\\
&\le |Q|^{-\frac{\alpha}{n}} \fint_Q \fint_Q |x - y|^\alpha\,dx\,dy
\sup_{\genfrac{}{}{0pt}{}{x,\,y \in Q}{x \ne y}} \frac{|f(x) - f(y)|}{|x - y|^\alpha} \notag \\
&\lesssim \sup_{\genfrac{}{}{0pt}{}{x,\,y \in Q}{x \ne y}} \frac{|f(x) - f(y)|}{|x - y|^\alpha},
\end{align}
where the implicit positive constant depends only on $n$.

Taking the supremum over all cubes $Q \subset \mathbb{R}^n$, we conclude that
\[
\|f\|_{\mathrm{BMO}_\alpha}
\lesssim \|f\|_{\Lambda_\alpha}
:= \sup_{{x,\,y \in \mathbb{R}^n,\,x \ne y}}
\frac{|f(x) - f(y)|}{|x - y|^\alpha}.
\]

Conversely, for any $\alpha \in (0,1)$, classical results due to Campanato \cite{c64} and Meyers \cite{m64} imply that
there exists a continuous representative $g$ of $f$ such that
$f = g$ almost everywhere on $\mathbb{R}^n$ and
$\|f\|_{\mathrm{BMO}_\alpha}
\gtrsim \|g\|_{\Lambda_\alpha}.$

This shows that $\mathrm{BMO}_\alpha = \Lambda_\alpha$ in the sense of equivalent norms and equality almost everywhere. Furthermore, by adapting ideas from~\cite{mo24}, we obtain the corresponding equivalences for the vanishing subspaces of $\mathrm{BMO}_\alpha$.

\begin{proposition}\label{prop-lip}
Let $\alpha \in (0,1)$. Then the following statements hold.
\begin{itemize}
\item[\rm(i)] $f \in \mathrm{VMO}_\alpha$ if and only if
$f \in \mathrm{BMO}_\alpha$ and
\[
\lim_{a \to 0^+} \sup_{{x,\,y \in \mathbb{R}^n,\,0 < |x - y| \le a}}
\frac{|f(x) - f(y)|}{|x - y|^\alpha} = 0.
\]

\item[\rm(ii)] $f \in \mathrm{XMO}_\alpha$ if and only if
$f \in \mathrm{VMO}_\alpha$ and, for any given cube $Q$,
\[
\lim_{z \to \infty} \sup_{{x,\,y \in Q+z, \, x \ne y}}
\frac{|f(x) - f(y)|}{|x - y|^\alpha} = 0.
\]

\item[\rm(iii)] $f \in \mathrm{CMO}_\alpha$ if and only if
$f \in \mathrm{XMO}_\alpha$ and
\[
\lim_{a \to \infty} \sup_{{x,\,y \in \mathbb{R}^n,\, |x - y| \ge a}}
\frac{|f(x) - f(y)|}{|x - y|^\alpha} = 0.
\]
\end{itemize}
\end{proposition}

\begin{proof}

To show the equivalence \textnormal{(i)},
from the oscillation--Lipschitz inequality \eqref{osc-lip},
it follows that, for any $a\in(0,\infty)$,
\[
\sup_{\ell(Q)\le a} \mathcal{O}_\alpha(f;Q)
\le \sup_{{x,\,y\in\mathbb{R}^n,\,0<|x-y|\le \sqrt{n}a}}
\frac{|f(x)-f(y)|}{|x-y|^\alpha}.
\]
This implies the ``if'' part of (i) by letting $a\to0^+$.

Conversely, let $f \in \mathrm{VMO}_\alpha$.
Then
Lemma \ref{lem:250509-1} implies the ``only if'' part of (i) by letting $a\to0^+$ and hence finishes the proof of (i).

Next, we prove (ii). The ``if'' part follows immediately from \eqref{osc-lip}.

Conversely, let \( Q \) be any given cube in \( \mathbb{R}^n \). Then
there exists a positive constant $C_0$ such that
\begin{align*}
\sup_{\genfrac{}{}{0pt}{}{x,\,y \in Q+z}{x \neq y}} \frac{|f(x) - f(y)|}{|x - y|^\alpha}
&\le C_0 \sup_{\genfrac{}{}{0pt}{}{x,\,y \in Q+z}{x \neq y}}
\sup_{Q \subset Q_0} \mathcal{O}_\alpha(f; Q_0)
\le C_0\sup_{Q \subset Q+z} \mathcal{O}_\alpha(f; Q)
\end{align*}
for any \( z \in \mathbb{R}^n \),
where the last inequality holds because we can choose \( Q_0 \subset Q+z \) as in the proof of (i).

Choose a cube \( Q(z) \subset Q+z \) such that the supremum is almost attained,
that is, choose a cube $Q(z)$ that satisfies
\[
\sup_{{x,\,y \in Q+z,\, x \neq y}} \frac{|f(x) - f(y)|}{|x - y|^\alpha}
\le 2C_0 \mathcal{O}_\alpha(f; Q(z)).
\] Then
\begin{align}\label{Q(z)}
\sup_{\genfrac{}{}{0pt}{}{x,\,y \in Q+z}{x \neq y}} \frac{|f(x) - f(y)|}{|x - y|^\alpha}
\lesssim \mathcal{O}_\alpha(f; Q(z))
\lesssim \left[ \frac{|Q+z|}{|Q(z)|} \right]^{1 + \frac{\alpha}{n}} \mathcal{O}_\alpha(f; Q+z).
\end{align}

Given any \( \epsilon \in (0, \infty) \), by Definition~\ref{defBVCXMO}(ii), there exists \( a_0 \in (0, \ell(Q)) \) such that, for any cube \( \widetilde{Q} \) with \( \ell(\widetilde{Q}) \le a_0 \),
\begin{align}\label{small}
\mathcal{O}_\alpha(f; \widetilde{Q}) < \epsilon.
\end{align}
Moreover, by Definition~\ref{defBVCXMO}(iii), there exists \( M \in (0, \infty) \) such that, for any \( z \in \mathbb{R}^n \) with \( |z| > M \),
\begin{align}\label{250513}
\mathcal{O}_\alpha(f; Q+z) < \left[ \frac{\ell(Q)}{a_0} \right]^{-(n+\alpha)} \epsilon.
\end{align}

If \( |z| > M \) and \( \ell(Q(z)) \le a_0 \), then, by \eqref{Q(z)} and \eqref{small},
\begin{align}\label{ii1}
\sup_{{x,\,y \in Q+z},\,{x \neq y}} \frac{|f(x) - f(y)|}{|x - y|^\alpha}
\lesssim \mathcal{O}_\alpha(f; Q(z)) < \epsilon.
\end{align}

If \( |z| > M \) and \( \ell(Q(z)) \ge a_0 \), then, using \eqref{Q(z)} and \eqref{250513},
\begin{align}\label{ii2}
\sup_{\genfrac{}{}{0pt}{}{x,\,y \in Q+z}{x \neq y}} \frac{|f(x) - f(y)|}{|x - y|^\alpha}
\lesssim \left[ \frac{|Q+z|}{|Q(z)|} \right]^{1 + \frac{\alpha}{n}} \mathcal{O}_\alpha(f; Q+z)
\le \left[ \frac{\ell(Q)}{a_0} \right]^{n + \alpha} \mathcal{O}_\alpha(f; Q+z) < \epsilon.
\end{align}

Combining \eqref{ii1} and \eqref{ii2}, we complete the proof of the ``only if'' part and thus finish the proof of (ii).

Now, we prove the equivalence in (iii).
In this case, the ``if'' part does not follow directly from \eqref{osc-lip}, and we need to modify the argument.

Let \( f \in \mathrm{XMO}_{\alpha} \) satisfy the limit condition in (iii).
Then, for any given \( \epsilon \in (0, \infty) \), there exists \( R \in (0, \infty) \) such that
\begin{align}\label{xyfar}
\sup_{{x,\,y\in\mathbb{R}^n,\,|x - y| \ge R}} \frac{|f(x) - f(y)|}{|x - y|^\alpha} < \epsilon.
\end{align}

Suppose \( Q \) is a cube such that
\begin{align}\label{Qlarge}
\ell(Q) \ge R
\quad \text{and} \quad
\left[ \frac{R}{\ell(Q)} \right]^{\alpha}
\sup_{\genfrac{}{}{0pt}{}{x,\,y \in \mathbb{R}^n}{0 < |x - y| \le R}} \frac{|f(x) - f(y)|}{|x - y|^\alpha} < \epsilon.
\end{align}

We estimate
\begin{align*}
|Q|^{-\frac{\alpha}{n}} \fint_Q |f(x) - \langle f \rangle_Q|\,dx
&\le |Q|^{-\frac{\alpha}{n}} \fint_Q \fint_Q |f(x) - f(y)|\,dx\,dy \\
&= |Q|^{-\frac{\alpha}{n} - 1} \fint_Q \int_Q |f(x) - f(y)|\,dx\,dy.
\end{align*}

We split the inner integral as follows:
\begin{align*}
\int_Q |f(x) - f(y)|\,dx
= \int_{Q \setminus B(y, R)} |f(x) - f(y)|\,dx + \int_{Q \cap B(y, R)} \cdots.
\end{align*}

For the first term, we use \eqref{xyfar} to obtain
\begin{align*}
\int_{Q \setminus B(y, R)} |f(x) - f(y)|\,dx
\le \epsilon \int_{Q \setminus B(y, R)} |x - y|^\alpha\,dx.
\end{align*}

For the second term, we apply the bound from \eqref{Qlarge} to find that
\begin{align*}
\int_{Q \cap B(y, R)} |f(x) - f(y)|\,dx
&\le \sup_{\substack{0 < |x - y| \le R}} \frac{|f(x) - f(y)|}{|x - y|^\alpha} \int_{Q \cap B(y, R)} |x - y|^\alpha\,dx\\
&\le \varepsilon \int_{Q \cap B(y, R)} |x - y|^\alpha\,dx.
\end{align*}

Combining both parts and integrating over \( y \in Q \), we obtain
\begin{align*}\fint_Q |f(x) - \langle f \rangle_Q|\,dx
&\le \fint_Q \left[
\int_{Q \setminus B(y, R)} \epsilon |x - y|^\alpha\,dx\right.\\
&\quad\left.+ \sup_{0 < |x - y| \le R} \frac{|f(x) - f(y)|}{|x - y|^\alpha} \int_{Q \cap B(y, R)} |x - y|^\alpha\,dx
\right] dy.
\end{align*}

Since \( |Q| = \ell(Q)^n \ge R^n\), the integrals can be estimated by
\[
\int_{Q \setminus B(y, R)} |x - y|^\alpha\,dx \le \ell(Q)^{\alpha + n}
\ \text{ and }\
\int_{Q \cap B(y, R)} |x - y|^\alpha\,dx \le \ell(Q)^\alpha R^{n}.
\]

Therefore,
\begin{align}\label{osc-diff}
|Q|^{-\frac{\alpha}{n}} \fint_Q |f(x) - \langle f \rangle_Q|\,dx
&\le \epsilon + \left[ \frac{R}{\ell(Q)} \right]^{\alpha} \sup_{0<|x - y| \le R} \frac{|f(x) - f(y)|}{|x - y|^\alpha}
< 2\epsilon,
\end{align}
which implies that \( f \in \mathrm{CMO}_\alpha \), as required.
This shows the ``if'' part of (iii) due to the arbitrariness of $\epsilon$.

Conversely, we also need to modify the estimate \eqref{fx-fy} to show the ``only if'' part.
Let $f\in \mathrm{CMO}_{\alpha}$.
Then, from Definition \ref{defBVCXMO}(iv), it follows that there exists $K\in\mathbb{N}$
such that, for any cube $Q$ with $\ell(Q)\ge 2^K$,
\begin{align}\label{osc-large}
\mathcal{O}_\alpha(f;Q)<\epsilon.
\end{align}
Let $M_0$ be the smallest integer such that
\begin{align*}
2^{(M_0-K)\alpha}>\epsilon^{-1}\sup_{\ell(Q)\le2^K}\mathcal{O}_\alpha(f;Q).
\end{align*}
For any given $x$, $y\in\mathbb{R}^n$ with $2^M\le|x-y|<2^{M+1}$ and $M_0\le M\in\mathbb{N}$,
define $Q_0$, $Q_k(x)$, and $Q_k(y)$ the same
as in the proof of Lemma \ref{lem:250509-1}.
Then, for any $k\in\{0,\dots,M-K\}$,
\begin{equation*}
\ell(Q_{k}(x))=2^{-k}\ell(Q_0)\ge 2^{-(M-K)}2^{M}=2^K,
\end{equation*}
which, together with \eqref{osc-large}, implies that
$\mathcal{O}_\alpha(f;Q_{k}(x))<\epsilon.$

We estimate the sum in \eqref{fx-fy} by splitting it as the local part
and the non-local, similar to \eqref{osc-diff}. First, observe that
\begin{equation*}
\left|\sum_{k=0}^\infty \left[\langle f\rangle_{Q_{k+1}(x)} - \langle f\rangle_{Q_k(x)}\right]\right|
\le \sum_{k=0}^\infty \left|\langle f\rangle_{Q_{k+1}(x)} - \langle f\rangle_{Q_k(x)}\right|.
\end{equation*}
Next, for each term in the sum,
keeping in mind that $Q_{k}(x) \supset Q_{k+1}(x)$, we apply the standard inequality
\begin{equation*}
\left|\langle f\rangle_{Q_{k+1}(x)} - \langle f\rangle_{Q_k(x)}\right|
\le 2^n\fint_{Q_k(x)} |f(x) - \langle f\rangle_{Q_k(x)}|\,dx.
\end{equation*}
This yields the estimate
\begin{equation*}
\sum_{k=0}^\infty \left|\langle f\rangle_{Q_{k+1}(x)} - \langle f\rangle_{Q_k(x)}\right|
\le 2^n\sum_{k=0}^\infty \fint_{Q_k(x)} |f(x) - \langle f\rangle_{Q_k(x)}|\,dx.
\end{equation*}

We now split the infinite sum into two parts:
\begin{align*}
\sum_{k=0}^\infty \fint_{Q_k(x)} |f(x) - \langle f\rangle_{Q_k(x)}|\,dx
&= \sum_{k=0}^{M-K} \fint_{Q_k(x)} |f(x) - \langle f\rangle_{Q_k(x)}|\,dx
+ \sum_{k=M-K+1}^\infty \cdots.
\end{align*}

We estimate each part separately. Since \(f \in \mathrm{BMO}_\alpha\), the local oscillation over small cubes satisfies
\begin{equation*}
\fint_{Q_k(x)} |f(x) - \langle f\rangle_{Q_k(x)}|\,dx < |Q_k(x)|^{\frac{\alpha}{n}}\epsilon \quad \text{for any } k\in\{0, \dots, M-K\}.
\end{equation*}
Thus,
\begin{equation*}
\sum_{k=0}^{M-K} \fint_{Q_k(x)} |f(x) - \langle f\rangle_{Q_k(x)}|\,dx
< \sum_{k=0}^{M-K} |Q_k(x)|^{\frac{\alpha}{n}}\epsilon.
\end{equation*}
Since \(\ell(Q_k(x)) \sim 2^{-k}|x-y|\), it follows that \(|Q_k(x)| \sim 2^{-kn}|x-y|^n\).
Therefore,
\begin{equation*}
\sum_{k=0}^{M-K} |Q_k(x)|^{\frac{\alpha}{n}}\epsilon \sim \sum_{k=0}^{M-K} 2^{-k\alpha}|x-y|^{\alpha}\epsilon
= |x-y|^{\alpha}\epsilon \sum_{k=0}^{M-K} 2^{-k\alpha}.
\end{equation*}
But this sum is geometric with bounded length (independent of \(x,y\)), and thus
$\sum_{k=0}^{M-K} 2^{-k\alpha} \sim 1.$
Consequently, we obtain
\begin{equation*}
\sum_{k=0}^{M-K} \fint_{Q_k(x)} |f(x) - \langle f\rangle_{Q_k(x)}|\,dx
< \epsilon |x-y|^{\alpha}.
\end{equation*}

For the remaining tail summation \(k \ge M-K+1\), we use the uniform bound
\begin{equation*}
\fint_{Q_k(x)} |f(x) - \langle f\rangle_{Q_k(x)}|\,dx
\le \sup_{\ell(Q)\le2^K} \mathcal{O}_\alpha(f;Q)  |Q_k(x)|^{\frac {\alpha}n}.
\end{equation*}
So,
\begin{equation*}
\sum_{k=M-K+1}^\infty \fint_{Q_k(x)} |f(x) - \langle f\rangle_{Q_k(x)}|\,dx
\le \sup_{\ell(Q)\le2^K} \mathcal{O}_\alpha(f;Q)
\sum_{k=M-K+1}^\infty |Q_k(x)|^{\frac {\alpha}n}.
\end{equation*}
As shown above, $|Q_k(x)|^{\frac{\alpha}n} \sim 2^{-k\alpha}|x-y|^{\alpha}$
and hence
\begin{equation*}
\sum_{k=M-K+1}^\infty |Q_k(x)|^{\frac {\alpha}n} \sim 2^{-(M-K)\alpha}|x-y|^{\alpha}.
\end{equation*}
Consequently,
\begin{equation*}
\sum_{k=M-K+1}^\infty \fint_{Q_k(x)} |f(x) - \langle f\rangle_{Q_k(x)}|\,dx
< 2^{-(M-K)\alpha} \sup_{\ell(Q)\le2^K} \mathcal{O}_\alpha(f;Q) |x-y|^\alpha.
\end{equation*}

Combining both estimates, we find
\begin{equation*}
\left|\sum_{k=0}^\infty \left[\langle f\rangle_{Q_{k+1}(x)} - \langle f\rangle_{Q_k(x)}\right]\right|
< \left[\epsilon + 2^{-(M-K)\alpha} \sup_{\ell(Q)\le2^K} \mathcal{O}_\alpha(f;Q)\right] |x-y|^\alpha.
\end{equation*}
For sufficiently large \(M\), the second term becomes smaller than \(\epsilon\), giving the final estimate
\begin{equation}\label{eq:250509-3}
\left|\sum_{k=0}^\infty \left[\langle f\rangle_{Q_{k+1}(x)} - \langle f\rangle_{Q_k(x)}\right]\right|
< 2\epsilon |x-y|^\alpha.
\end{equation}
Similarly, we also have
\[
\left|\sum_{k=0}^\infty \left[\langle f\rangle_{Q_{k+1}(y)}
-\langle f\rangle_{Q_{k}(y)}\right] \right|\lesssim\epsilon |x-y|^\alpha,
\]
which, combined with \eqref{fx-fy} and \eqref{eq:250509-3}, implies that
$|f(x)-f(y)|\lesssim\epsilon$
for any $x, y\in\mathbb{R}^n$ with $|x-y|\ge2^{M_0}$.
This shows the ``only if'' part of (iii) due to the arbitrariness of $\epsilon$,
which completes the proof of Proposition \ref{prop-lip}.
\end{proof}

\begin{remark}
Proposition~\ref{prop-lip} provides a new characterization of
$\mathrm{XMO}_{\alpha}$, as introduced in \cite{tyyMANN}.
\end{remark}

We now turn to the higher-order case by employing the convolution method, following the approach of \cite{gr85}.

\begin{proposition}\label{prop-zyg}
The following equivalences hold almost everywhere:
\begin{itemize}
\item[\rm(i)] $f \in \mathrm{V}\mathcal{L}_1$ if and only if
$f \in \mathcal{L}_1$ and
\[
\lim_{a \to 0^+} \sup_{{x,\,y \in \mathbb{R}^n,\, 0 < |y| \le a}}
\frac{|\Delta^2_y f(x)|}{|y|} = 0.
\]

\item[\rm(ii)] $f \in \mathrm{X}\mathcal{L}_1$ if and only if
$f \in \mathrm{V}\mathcal{L}_1$ and
\[
\lim_{z \to \infty} \sup_{{2|y| \le \ell(Q),\, x \in Q + z}}
\frac{|\Delta^2_y f(x)|}{|y|} = 0
\]
for any fixed cube $Q$;

\item[\rm(iii)] $f \in \mathrm{C}\mathcal{L}_1$ if and only if
$f \in \mathrm{X}\mathcal{L}_1$ and
\[
\lim_{a \to \infty} \sup_{{x,\,y \in \mathbb{R}^n,\,|y| \ge a}}
\frac{|\Delta^2_y f(x)|}{|y|} = 0.
\]
\end{itemize}
\end{proposition}

\begin{proof}
We first show (i).
Note that the “if” part follows directly from Lemma~\ref{lem-2diff}(ii) with $X = L^1$.
Hence, it remains to prove the “only if” part.

Let $\phi$ be a infinitely differentiable function supported in $B(\mathbf{0},1)$.
For any $t\in(0,\infty)$, define $\phi_t(\cdot):=\frac{1}{t^n}\phi(\frac{\cdot}{t})$.
Then $\{\phi_t\}_{t\in(0,\infty)}$
is an approximation of the identity, and hence we have
$f_0 := \lim_{t \to 0^+} f * \phi_t = f$ almost everywhere on $\mathbb{R}^n.$

For any $(x,t) \in \mathbb{R}^n\times(0,\infty)$,
define
\[
u_0(x,0) := f_0(x),\ u_0(x,t) := f * \phi_t(x),\
u_1(x,t) := -t \frac{\partial}{\partial t} u_0(x,t),
\]
and
\[
u_2(x,t) := \int_0^t s \frac{\partial^2}{\partial s^2} u_0(x,s) \, ds.
\]

Applying the Newton--Leibniz formula and integration by parts,
we obtain the following identity: for any $(x,t) \in \mathbb{R}^n\times(0,\infty)$,
\begin{align}\label{u012}
f_0(x) &= u_0(x,0)
= u_0(x,t) - \int_0^t \frac{\partial}{\partial s} u_0(x,s) \, ds \notag\\
&= u_0(x,t) - t \frac{\partial}{\partial t} u_0(x,t)
+ \int_0^t s \frac{\partial^2}{\partial s^2} u_0(x,s) \, ds \notag \\
&= u_0(x,t) + u_1(x,t) + u_2(x,t).
\end{align}

From the assumption
\[
\lim_{a \to 0^+} \sup_{\ell(Q) \le a} \mathcal{O}_1(f;Q) = 0,
\]
we deduce that, for any $\varepsilon\in(0,\infty)$, there exists $\delta\in(0,\infty)$ such that, for any $r\in(0,\delta]$ and any $x \in \mathbb{R}^n$,
\begin{align}\label{O1small}
\mathcal{O}_1(f; B(x,r)) < \varepsilon.
\end{align}

It remains to estimate the second-order differences of each $u_j$ with $j \in \{0,1,2\}$.
We first establish the following key estimate: for any $k \in \mathbb{Z}_+$ and $(x,t)\in\mathbb{R}^n\times(0,\infty)$,
\begin{align}\label{core}
\left| \frac{\partial^k}{\partial t^k} u_0(x,t) \right|
\lesssim t^{1-k} \mathcal{O}_1(f; B(x,t)).
\end{align}

Indeed, define $\widetilde{f}_x(y) := f(x - y)$ for any $y \in \mathbb{R}^n$ and
let $a_t := t^{k-1} \frac{\partial^k}{\partial t^k} \phi_t$ for any $t \in (0,\infty)$.
By \cite[p.\,301, Lemma 5.20]{gr85}, we conclude that $a_t$ satisfies
\[
\|a_t\|_{L^\infty}
\lesssim t^{-(n+1)}
\ \text{ and }\
\int_{{\mathbb R}^n}a_t(x)\, dx=\int_{{\mathbb R}^n}x_j a_t(x)\, dx=0
\]
for any $j\in\{1,2,\ldots,n\}$.
Since
\[
\frac{\partial^k}{\partial t^k} u_0(x,t)
=\int_{\mathbb{R}^n} f(x-y) \frac{\partial^k}{\partial t^k} \phi_t(y) \, dy
=t^{1-k}\int_{\mathbb{R}^n} \left[\widetilde{f}_x(y)
- P^1_{B(\mathbf{0},t)}(\widetilde{f}_x)(y)\right] a_t(y) \, dy,
\]
it follows that
\begin{align*}
\left| \frac{\partial^k}{\partial t^k} u_0(x,t) \right|
&\lesssim t^{1-k}\|a_t\|_{L^\infty}
\int_{B(\mathbf{0},t)} \left| \widetilde{f}_x(y)
- P^1_{B(\mathbf{0},t)}(\widetilde{f}_x)(y) \right| \, dy \notag\\
&\lesssim t^{1-k}|B(\mathbf{0},t)|^{-(1 + \frac{1}{n})}
\int_{B(\mathbf{0},t)} \left| \widetilde{f}_x(y)
- P^1_{B(\mathbf{0},t)}(\widetilde{f}_x)(y) \right| \, dy \notag\\
&= t^{1-k} \mathcal{O}_1(\widetilde{f}_x; B(\mathbf{0},t))
\lesssim t^{1-k} \mathcal{O}_1(f; B(x,t)).
\end{align*}

We now estimate each $u_j$ by using this bound.
Fix $h \in \mathbb{R}^n \setminus \{\mathbf{0}\}$ with $|h| < \delta$.
For any $s \in (0,|h|)$, we have $s \le |h| < \delta$.
Using \eqref{core} with $k=2$ and $t = |h|$, together with \eqref{O1small}, we obtain, for any $x \in \mathbb{R}^n$,
\begin{align}\label{u2xh}
|u_2(x,|h|)|
&= \left| \int_0^{|h|} s \frac{\partial^2}{\partial s^2} u_0(x,s) \, ds \right|
\le \int_0^{|h|} s \left| \frac{\partial^2}{\partial s^2} u_0(x,s) \right| \, ds \notag\\
&\lesssim \int_0^{|h|} \mathcal{O}_1(f; B(x,s)) \, ds
\lesssim \varepsilon |h|,
\end{align}
which implies
\begin{align}\label{u2}
\frac{|\Delta^2_h u_2(x,|h|)|}{|h|} \lesssim \varepsilon.
\end{align}

Similarly, using \eqref{core} with $k \in \{0, 1\}$ and $t = |h|$, we find
\begin{align}\label{u01}
|u_0(x,|h|)| + |u_1(x,|h|)|
\lesssim |h| \mathcal{O}_1(f; B(x,|h|)),
\end{align}
which gives
\begin{align}\label{u12diff}
\frac{|\Delta^2_h u_0(x,|h|)|}{|h|} + \frac{|\Delta^2_h u_1(x,|h|)|}{|h|}
\lesssim \varepsilon.
\end{align}

Combining \eqref{u012}, \eqref{u2}, and \eqref{u12diff}, we conclude that
\begin{align}\label{u012diff}
\frac{|\Delta^2_h f_0(x)|}{|h|}
\le \frac{|\Delta^2_h u_0(x,|h|)|}{|h|}
+ \frac{|\Delta^2_h u_1(x,|h|)|}{|h|}
+ \frac{|\Delta^2_h u_2(x,|h|)|}{|h|}
\lesssim \varepsilon.
\end{align}
This finishes the proof of the “only if” part of (i).

Now we prove the equivalence in (ii).

The “if” part follows directly from Lemma~\ref{lem-2diff}(ii) with $X = L^1$ and $Q$ replaced by $Q + z$.

To prove the “only if” part, choose $\varepsilon\in(0,\infty)$ and $\delta\in(0,\infty)$ as in part (i), and decompose
\[
f_0(x) = u_0(x,t) + u_1(x,t) + u_2(x,t),
\text{ for any }(x,t) \in \mathbb{R}^n \times (0,\infty),
\]
as in \eqref{u012}.
It suffices to consider $h \in \mathbb{R}^n$ with $|h| \ge \delta$
because the case $|h| < \delta$ has already been treated in (i).

Choose $M\in(0,\infty)$ sufficiently large such that,
for any $x \in \mathbb{R}^n$ with $|x| > M$,
\begin{align}\label{MLarge}
\left( \frac{|h|}{\delta} \right)^{n+1} \mathcal{O}_1\left(f; B(x,|h|) \right) < \varepsilon.
\end{align}

Using this and arguments similar to those in the estimation of \eqref{u2xh},
we obtain, for any $|x| > M$,
\begin{align*}
|u_2(x,|h|)|
&= \left| \int_0^{|h|} s \frac{\partial^2}{\partial s^2} u(x,s)\, ds \right| \\
&\lesssim \int_0^{|h|} \mathcal{O}_1(f; B(x,s))\, ds
= \int_0^\delta \mathcal{O}_1(f; B(x,s))\, ds
+ \int_\delta^{|h|} \mathcal{O}_1(f; B(x,s))\, ds \\
&\lesssim \int_0^\delta \varepsilon\, ds
+ \int_\delta^{|h|} \left( \frac{|h|}{\delta} \right)^{n+1}
\mathcal{O}_1(f; B(x,|h|))\, ds \\
&\lesssim \delta \varepsilon + (|h| - \delta)
\left( \frac{|h|}{\delta} \right)^{n+1} \mathcal{O}_1(f; B(x,|h|))\lesssim \varepsilon |h|,
\end{align*}
and hence \eqref{u2} remains valid.

Moreover, combining \eqref{MLarge} with \eqref{u01}, we find that
\[
|u_0(x,|h|)| + |u_1(x,|h|)| \lesssim |h| \mathcal{O}_1(f; B(x,|h|)) \lesssim \varepsilon |h|,
\]
which implies that both \eqref{u12diff} and \eqref{u012diff} still hold.

This finishes the proof of the “only if” part of (ii).

Let $Q$ be a cube centered at $x_0$ with edge length $2r$, where $r$ is sufficiently large to be determined later. As in Lemma~\ref{lem-2diff}, we estimate the oscillation using the second-order difference.

Given any $\varepsilon \in (0,\infty)$, choose $R \in (1,\infty)$ such that
\begin{align}\label{R-large}
\sup_{{x,\,z \in \mathbb{R}^n,\,|z| \ge R}} \frac{|\Delta^2_z f(x)|}{|z|} < \varepsilon.
\end{align}
Next, choose $r \in (R, \infty)$ large enough so that
\begin{align}\label{r-large}
\frac{R}{r} \sup_{{x,\,z \in \mathbb{R}^n,\,|z| \le R}} \frac{|\Delta^2_z f(x)|}{|z|} < \varepsilon.
\end{align}

Let $\phi$ and $\psi$ be as in the proof of Lemma~\ref{lem-2diff}. Using the estimate from \eqref{Dpsi}, together with \eqref{R-large} and \eqref{r-large}, we obtain, for any multi-index $\gamma \in (\mathbb{Z}_+)^n$ with $|\gamma| = 2$,
\begin{align*}
|2D^\gamma \psi(x)|
&\le r^{-1} \int_{B(\mathbf{0},r) \setminus B(\mathbf{0},R)} |(D^\gamma \phi)_r(y)|\, \frac{|y|}{r}
\sup_{\substack{x,\,z \in \mathbb{R}^n\\ |z| \ge R}} \frac{|\Delta^2_z f(x)|}{|z|} \,dy \\
&\quad + r^{-1} \int_{B(\mathbf{0},R)} |(D^\gamma \phi)_r(y)|\, \frac{|y|}{r}
\sup_{\substack{x,\,z \in \mathbb{R}^n\\ |z| \le R}} \frac{|\Delta^2_z f(x)|}{|z|} \,dy \\
&\le r^{-1} \int_{B(\mathbf{0},r) \setminus B(\mathbf{0},R)} |(D^\gamma \phi)_r(y)|
\sup_{\substack{x,\,z \in \mathbb{R}^n\\ |z| \ge R}} \frac{|\Delta^2_z f(x)|}{|z|} \,dy \\
&\quad + r^{-1} \int_{B(\mathbf{0},R)} |(D^\gamma \phi)_r(y)|\, \frac{R}{r}
\sup_{\substack{x,\,z \in \mathbb{R}^n\\ |z| \le R}} \frac{|\Delta^2_z f(x)|}{|z|} \,dy \\
&\le 2r^{-1} \|D^\gamma \phi\|_{L^1} \varepsilon \sim r^{-1} \varepsilon.
\end{align*}

By the Taylor theorem, this implies that
$|\psi(x) - P_1(x)| \lesssim r^{-1} \varepsilon |x - x_0|^2 \lesssim r \varepsilon,$
where $P_1$ is the first-order Taylor polynomial of $f$ at $x_0$.

Moreover, applying the argument  used in the proof of Lemma~\ref{lem-2diff}(i),
we also conclude that
$|f(x) - \psi(x)| \lesssim r \varepsilon,$
which yields
\begin{align}\label{f-P1}
|f(x) - P_1(x)| \lesssim r \varepsilon.
\end{align}
This in turn gives us
\[
\mathcal{O}_1(f; Q) \lesssim r^{-(n+1)} \int_Q |f(x) - P_1(x)|\, dx \lesssim \varepsilon,
\]
which completes the proof of the ``if'' part of (iii).
It remains to prove the ``only if'' part of (iii).
We focus on the case where $|h|$ is sufficiently large, with the precise bound to be determined later.

Choose $\varepsilon,\delta\in(0,\infty)$ as in part (i), and decompose
\begin{equation*}
  f_0(x) = u_0(x,t) + u_1(x,t) + u_2(x,t),
\text{ for any } (x,t) \in \mathbb{R}^n \times (0,\infty),
\end{equation*}
as in \eqref{u012}.

By the assumption, there exists $A \in (0,\infty)$ such that, for any $s \in [A,\infty)$,
\begin{equation}\label{O1large}
  \mathcal{O}_1(f; B(x,s)) < \varepsilon.
\end{equation}

Now, for any $h \in \mathbb{R}^n$ with
$|h| > (A - \delta)\left(\frac{A}{\delta}\right)^{n+1},$
we estimate $u_2(x,|h|)$ as follows
\begin{align*}
  |u_2(x,|h|)| &\leq \int_0^{|h|} s\left|\frac{\partial^2}{\partial s^2} u(x,s)\right| \,ds \\
  &\lesssim \left( \int_0^{\delta} + \int_{\delta}^A + \int_A^{|h|} \right) \mathcal{O}_1(f; B(x,s)) \,ds \\
  &\lesssim \delta \varepsilon + (A - \delta)\left(\frac{A}{\delta}\right)^{n+1} \mathcal{O}_1(f; B(x,A)) + (|h| - A)\varepsilon\lesssim \varepsilon |h|,
\end{align*}
which implies that \eqref{u2} still holds for large $|h|$.

The estimates for $u_0(x,|h|)$ and $u_1(x,|h|)$ in the case of
large $|h|$ directly follow from \eqref{core} and \eqref{O1large}.
Hence, \eqref{u12diff} and consequently \eqref{u012diff} still hold.

This finishes the proof of the ``only if'' part of (iii) and thus Proposition~\ref{prop-zyg}.
\end{proof}

\section{Proof of Theorem \ref{thm1} and Applications to Specific Function Spaces}\label{s4}

In this section, we first prove Theorem \ref{thm1} in Subsection \ref{s4.1}
and then apply Theorem \ref{thm1} to some specific function spaces in Subsection \ref{s4.2}.

\subsection{Proof of Theorem \ref{thm1}}\label{s4.1}

In this subsection, we prove Theorem \ref{thm1} by distinguishing three cases:
$\alpha = 0$, $\alpha \in (0,1)$, and $\alpha = 1$ separately.

\begin{proof}[Proof of Theorem \ref{thm1}]
We begin with the case $\alpha = 0$.
By Lemma \ref{lem-oscX}, it suffices to show that
\[
\mathrm{YMO} := \mathrm{Y}\mathcal{L}_0
\subset \mathrm{Y}\mathcal{L}_{0,X} =: \mathrm{YMO}_X
\]
for any $\mathrm{Y} \in \{\mathrm{V,\,X,\,C}\}$.
We treat the three cases separately.

Let $f \in \mathrm{VMO}$.
By Lemma \ref{lem-vcxmo}(i), for any $\epsilon \in (0, \infty)$, there exists a function
$g \in C_{\mathrm{u}} \cap \mathrm{BMO}$
such that $\|f - g\|_{\mathrm{BMO}} < \epsilon$.
Applying this and Proposition \ref{bmox}, we obtain
$\|f - g\|_{\mathrm{BMO}_X} \lesssim \epsilon.$
Therefore,
\begin{align*}
\mathcal{O}_X(f; Q)
\le \mathcal{O}_X(f - g, Q) + \mathcal{O}_X(g, Q)
\le \|f - g\|_{\mathrm{BMO}_X} + \mathcal{O}_X(g, Q)
\lesssim \epsilon + \mathcal{O}_X(g, Q).
\end{align*}
Since $g \in C_{\mathrm{u}}$, it follows from Lemma \ref{lem-oscX} that
\[
\lim_{a \to 0^+} \sup_{|Q| \le a} \mathcal{O}_X(f; Q) \lesssim \epsilon.
\]
As $\epsilon\in(0,\infty)$ is arbitrary, we conclude $f \in \mathrm{VMO}_X$.

Next, let $f \in \mathrm{XMO}$.
By Lemma \ref{lem-vcxmo}(ii), for any $\epsilon \in (0, \infty)$, there exists a function
$g \in B_\infty \cap \mathrm{BMO}$
such that $\|f - g\|_{\mathrm{BMO}} < \epsilon$.
Since $B_\infty \subset C_{\mathrm{u}}$, we have
\[
\lim_{a \to 0^+} \sup_{|Q| \le a} \mathcal{O}_X(g, Q) = 0.
\]
Moreover, by the definition of $B_\infty$, Lemma \ref{lem-oscX},
and the mean value theorem, we obtain
\[
\mathcal{O}_X(g, Q + x) \le \sup_{y, z \in Q + x} |g(y) - g(z)|
\lesssim \sup_{\xi \in Q + x} |\nabla g(\xi)| \, \ell(Q + x) \to 0
\]
as $x \to \infty$. Hence, $f \in \mathrm{XMO}_X$.

Finally, we consider $f \in \mathrm{CMO}$.
By a similar density argument and the fact that
$C_{\rm c}^\infty \subset \mathrm{CMO}$, it suffices to show
$\mathcal{O}_X(g, Q) \to 0$ as $|Q| \to \infty$
for any $g \in C_{\rm c}^\infty$.
Indeed, we have
\[
\mathcal{O}_X(g, Q)
= \frac{1}{\|\mathbf{1}_Q\|_X} \left\| (g - \langle g \rangle_Q) \mathbf{1}_Q \right\|_X
\le \frac{\|g\|_X}{\|\mathbf{1}_Q\|_X} + \frac{\|g\|_{L^1}}{|Q|} \to 0
\]
as $|Q| \to \infty$, by Lemma \ref{const} and the Fatou property of $X$.
Thus, $f \in \mathrm{CMO}_X$, completing the case $\alpha = 0$.

Next, we consider the case $\alpha \in (0,1)$.
By Lemma \ref{lem-oscX}, it suffices to prove that
\begin{align}\label{YMO}
\mathrm{YMO}_\alpha := \mathrm{Y}\mathcal{L}_{\alpha}
\subset \mathrm{Y}\mathcal{L}_{\alpha,X} =: \mathrm{YMO}_{\alpha,X}
\end{align}
for any $\mathrm{Y} \in \{\mathrm{V}, \mathrm{X}, \mathrm{C}\}$.

From Lemma \ref{lem-oscX} and Proposition \ref{prop-lip}, we deduce that
\begin{align*}
\lim_{a \to 0^+} \sup_{\ell(Q) \le a} \mathcal{O}_{\alpha, X}(f; Q)
&\lesssim \lim_{a \to 0^+} \sup_{\ell(Q) \le a}
\sup_{\genfrac{}{}{0pt}{}{x,\,y \in Q}{x \neq y}}
\frac{|f(x) - f(y)|}{|x - y|^\alpha} \\
&\lesssim \lim_{a \to 0^+}
\sup_{\genfrac{}{}{0pt}{}{x,\,y \in {\mathbb R}^n}{0 < |x - y| \le \sqrt{n}a}}
\frac{|f(x) - f(y)|}{|x - y|^\alpha}
= 0
\end{align*}
and
\begin{align*}
\lim_{z \to \infty} \mathcal{O}_{\alpha,X}(f; Q+z)
&\le \lim_{z \to \infty} \sup_{\genfrac{}{}{0pt}{}{x,\,y \in Q+z}{x \neq y}} \frac{|f(x) - f(y)|}{|x - y|^\alpha}
= 0.
\end{align*}
Hence, \eqref{YMO} holds for $\mathrm{Y} \in \{\mathrm{V}, \mathrm{X}\}$.

To complete the proof, it remains to show that
$\mathrm{CMO}_\alpha \subset \mathrm{CMO}_{\alpha,X}.$
Let $Q$ be a large cube as in \eqref{Qlarge}.
It suffices to prove that
\begin{equation*}
\frac{|Q|^{-\frac{\alpha}{n}}}{\|\mathbf{1}_Q\|_X}
\left\| \left(f - \langle f \rangle_Q\right) \mathbf{1}_Q \right\|_X \to 0
\quad \text{as }\ \ell(Q) \to \infty.
\end{equation*}

We first estimate
\begin{align*}
\frac{|Q|^{-\frac{\alpha}{n}}}{\|\mathbf{1}_Q\|_X}
\left\| \left(f - \langle f \rangle_Q\right) \mathbf{1}_Q \right\|_X
&\le \frac{|Q|^{-\frac{\alpha}{n}}}{\|\mathbf{1}_Q\|_X}
\left\| \int_Q |f(\cdot) - f(y)| \, dy \, \mathbf{1}_Q \right\|_X  \\
&= \frac{|Q|^{-\frac{\alpha}{n} - 1}}{\|\mathbf{1}_Q\|_X}
\left\| \int_Q |f(\cdot) - f(y)| \, dy \, \mathbf{1}_Q \right\|_X.
\end{align*}

We split the integral over $Q$ into near and far regions
by using a fixed parameter $R\in(0,\infty)$ to find that
\begin{align*}
\frac{|Q|^{-\frac{\alpha}{n}}}{\|\mathbf{1}_Q\|_X}
\left\| \left(f - \langle f \rangle_Q\right) \mathbf{1}_Q \right\|_X
&\le \frac{|Q|^{-\frac{\alpha}{n} - 1}}{\|\mathbf{1}_Q\|_X}
\left\| \int_{Q \setminus B_R(y)} |f(\cdot) - f(y)| \, dy \, \mathbf{1}_Q \right\|_X \notag \\
&\quad + \frac{|Q|^{-\frac{\alpha}{n} - 1}}{\|\mathbf{1}_Q\|_X}
\left\| \int_{Q \cap B_R(y)} |f(\cdot) - f(y)| \, dy \, \mathbf{1}_Q \right\|_X.
\end{align*}

On $Q \setminus B_R(y)$, we use that $f \in \mathrm{CMO}_\alpha$ implies that,
for any $\epsilon\in(0,\infty)$, there exists $R = R(\epsilon)$ such that
$|f(x) - f(y)| \le \epsilon |x - y|^\alpha$ whenever $|x - y| > R.$
Hence,
\begin{align*}
\frac{|Q|^{-\frac{\alpha}{n}}}{\|\mathbf{1}_Q\|_X}
\left\| \left(f - \langle f \rangle_Q\right) \mathbf{1}_Q \right\|_X
&\le \frac{|Q|^{-\frac{\alpha}{n} - 1}}{\|\mathbf{1}_Q\|_X}
\left\| \int_{Q \setminus B_R(y)} \epsilon |\cdot - y|^\alpha \, dy \, \mathbf{1}_Q \right\|_X \notag \\
&\quad + \frac{|Q|^{-\frac{\alpha}{n} - 1}}{\|\mathbf{1}_Q\|_X}
\left\| \int_{Q \cap B_R(y)} |\cdot - y|^\alpha \, dy \, \mathbf{1}_Q \right\|_X
\sup_{\genfrac{}{}{0pt}{}{x,\,y \in \mathbb{R}^n}{0 < |x - y| \le R}} \frac{|f(x) - f(y)|}{|x - y|^\alpha}.
\end{align*}

The first integral is bounded by $\epsilon \int_Q |\cdot - y|^\alpha \, dy \lesssim \epsilon |Q|^{\frac{\alpha}{n} + 1}$. The second integral is supported in a ball of radius $R$, so it is bounded by $R^{\alpha}\ell(Q)^n$. Thus,
\begin{align*}
\frac{|Q|^{-\frac{\alpha}{n}}}{\|\mathbf{1}_Q\|_X}
\left\| \left(f - \langle f \rangle_Q\right) \mathbf{1}_Q \right\|_X
&\lesssim \frac{|Q|^{-\frac{\alpha}{n} - 1}}{\|\mathbf{1}_Q\|_X}
\left\| \epsilon |Q|^{\frac{\alpha}{n} + 1} \mathbf{1}_Q \right\|_X  \\
&\quad + \frac{|Q|^{-\frac{\alpha}{n} - 1}}{\|\mathbf{1}_Q\|_X}
\left\| R^{\alpha + n} \mathbf{1}_Q \right\|_X
\sup_{\genfrac{}{}{0pt}{}{x,\,y \in \mathbb{R}^n}{0 < |x - y| \le R}} \frac{|f(x) - f(y)|}{|x - y|^\alpha}. \notag
\end{align*}

Simplifying, we obtain
\begin{align*}
\frac{|Q|^{-\frac{\alpha}{n}}}{\|\mathbf{1}_Q\|_X}
\left\| \left(f - \langle f \rangle_Q\right) \mathbf{1}_Q \right\|_X
&\lesssim \epsilon + \left[ \frac{R}{\ell(Q)} \right]^{\alpha}
\sup_{\genfrac{}{}{0pt}{}{x,\,y \in \mathbb{R}^n}{0 < |x - y| \le R}} \frac{|f(x) - f(y)|}{|x - y|^\alpha}.
\end{align*}

Using the assumptions that $f$ is uniformly H\"older continuous
of order $\alpha$ and $\ell(Q)$ is large,
we find that the second term is small. Choosing $\ell(Q) > R(\epsilon)$ large enough yields
\begin{equation*}
\frac{|Q|^{-\frac{\alpha}{n}}}{\|\mathbf{1}_Q\|_X}
\left\| \left(f - \langle f \rangle_Q\right) \mathbf{1}_Q \right\|_X
< 2\epsilon.
\end{equation*}

Since $\epsilon\in(0,\infty)$ was arbitrary, we conclude that $f \in \mathrm{CMO}_{\alpha,X}$, completing the proof of Theorem \ref{thm1} for $\alpha \in (0,1)$.

Now, we consider the case \( \alpha = 1 \).
By Lemma~\ref{lem-2diff}(iii), it suffices to prove that
\begin{align}
  \label{YL}
  \mathrm{Y}\mathcal{L}_{\alpha} \subset \mathrm{Y}\mathcal{L}_{\alpha,X}
\end{align}
for any \( \mathrm{Y} \in \{\mathrm{V},\,\mathrm{X},\,\mathrm{C}\} \).
Moreover, by combining Lemma~\ref{lem-2diff}(ii) and
both (ii) and (iii) of Proposition~\ref{prop-zyg},
we obtain \eqref{YL} for \( \mathrm{Y} \in \{\mathrm{V},\,\mathrm{X}\} \).
Thus, it remains to verify the inclusion
$\mathrm{C}\mathcal{L}_{\alpha} \subset \mathrm{C}\mathcal{L}_{\alpha,X}.$
Let \( f \in \mathrm{C}\mathcal{L}_{\alpha} \).
Then, by Proposition~\ref{prop-zyg}(iii), we have
\[
  \lim_{a \to \infty} \sup_{{x,\,y \in \mathbb{R}^n,\,|y| \ge a}} \frac{|\Delta^2_y f(x)|}{|y|} = 0.
\]
Let \( \varepsilon,r\in(0,\infty) \) be the same as in \eqref{r-large}.
Using Definition~\ref{def-bbfs}, \eqref{Cs}, and \eqref{f-P1}, we conclude that, for any cube \( Q \subset \mathbb{R}^n \) with edge length \( \ell(Q) \ge 2r \),
\begin{align*}
  \left\|\left[f - P^1_Q(f)\right]\mathbf{1}_Q\right\|_X
  &\le \left\|\left[f - P_1\right]\mathbf{1}_Q\right\|_X
      + \left\|\left[P^1_Q(P_1 - f)\right]\mathbf{1}_Q\right\|_X \\
  &\lesssim \left\|\left[f - P_1\right]\mathbf{1}_Q\right\|_X
      + \left\| \fint_Q |f(y) - P_1(y)|\,dy \, \mathbf{1}_Q \right\|_X \\
  &\lesssim |Q|^{1/n} \varepsilon \|\mathbf{1}_Q\|_X.
\end{align*}
Therefore,
\begin{align*}
  \frac{|Q|^{-1/n}}{\|\mathbf{1}_Q\|_X}
  \left\|\left[f - P^1_Q(f)\right]\mathbf{1}_Q\right\|_X
  \lesssim \varepsilon.
\end{align*}
Since \( \varepsilon\in(0,\infty) \) was arbitrary, this implies that
\( f \in \mathrm{C}\mathcal{L}_{\alpha,X} \),
which completes the proof of Theorem~\ref{thm1} in the case \( \alpha = 1 \).

Combining all three cases with Remark~\ref{rem-2}(iv), we conclude the proof of Theorem~\ref{thm1}.
\end{proof}
\begin{remark}
Another form of oscillation $\widetilde{\mathcal{O}}_{1,X}(f; Q)$
arises when $\alpha = 1$, defined by setting
\[
\widetilde{\mathcal{O}}_{1,X}(f; Q) := \frac{|Q|^{-1/n}}{\|\mathbf{1}_Q\|_X}
\left\| \left(f - \langle f \rangle_Q \right) \mathbf{1}_Q \right\|_X.
\]
In particular, when $X = L^1$, we omit the subscript $X$ and simply write $\widetilde{\mathcal{O}}_1(f; Q)$.

It was shown by Meyer~\cite{m64} that
\[
\sup_{\text{cube } Q} \widetilde{\mathcal{O}}_1(f; Q)
\sim \|f\|_{\mathrm{Lip}}
:= \sup_{\genfrac{}{}{0pt}{}{x,\,y \in \mathbb{R}^n}{x \neq y}} \frac{|f(x) - f(y)|}{|x - y|}.
\]
One can verify that both Lemma~\ref{lem-oscX} and Proposition~\ref{prop-lip} remain valid when setting $\alpha = 1$ and replacing $\mathcal{O}_{\alpha,X}$ with $\widetilde{\mathcal{O}}_{1,X}$. Consequently, Theorem~\ref{thm1} also holds with $\mathcal{O}_{1,X}$ replaced by $\widetilde{\mathcal{O}}_{1,X}$.
\end{remark}
\begin{remark}
In some recent studies \cite{dlyyz,dpgyyz,is17}, the assumption that the Hardy--Littlewood maximal operator $M$ is bounded on $X'$ can be replaced by weaker conditions, such as:
\begin{itemize}
\item[\rm (a)] the family of centered ball average operators $\{B_r\}_{r\in(0,\infty)}$ is uniformly bounded on $X$, where $B_r$ is defined as in \eqref{Mf};
\item[\rm (b)] the operator $M$ is \emph{weakly bounded} on $X'$, that is, there exists a positive constant $C$ such that, for any $f \in X'$ and any $\lambda\in(0,\infty)$,
$\| \mathbf{1}_{\{x \in \mathbb{R}^n : M(f)(x) > \lambda\}}\|_{X'} \le C \lambda^{-1} \|f\|_{X'}.$
\end{itemize}

With either of these conditions, various useful results, such as Lemmas \ref{reholder}, \ref{lem-oscX}, and \ref{lem-2diff}, still hold even without the full boundedness of $M$ on $X'$. However, in the context of Theorem \ref{thm1}, such replacements are not valid.

To see this, consider the case where $X = L^\infty$.
Then the space ${\rm BMO}_X$ coincides with the quotient space $L^\infty / \mathbb{C}$,
where $\mathbb{C}$ denotes the subspace of $L^\infty$ consisting of
functions equal almost everywhere to a constant.
Note that $M$ is unbounded on $X' = L^1$,
although both conditions (a) and (b) are easily verified in this case.

Indeed, for any cube $Q \subset \mathbb{R}^n$,
\begin{align*}
\mathcal{O}_X(f; Q)
&= \frac{1}{\|\mathbf{1}_Q\|_X} \left\| (f - \langle f \rangle_Q)\mathbf{1}_Q \right\|_X
= \left\| (f - \langle f \rangle_Q)\mathbf{1}_Q \right\|_{L^\infty} \\
&= \left\| \frac{1}{|Q|} \int_Q [f(\cdot) - f(y)] \, dy \cdot \mathbf{1}_Q \right\|_{L^\infty}
\le \operatorname*{ess\,sup}_{x, y \in Q} |f(x) - f(y)| \\
&\le \operatorname*{ess\,sup}_{x, y \in Q} \left( |f(x) - \langle f \rangle_Q| + |\langle f \rangle_Q - f(y)| \right) \\
&= 2 \operatorname*{ess\,sup}_{x \in Q} |f(x) - \langle f \rangle_Q|
= 2 \mathcal{O}_X(f; Q),
\end{align*}
and hence
\[
\|f\|_{{\rm BMO}_X} \sim \sup_Q \operatorname*{ess\,sup}_{x, y \in Q} |f(x) - f(y)|,
\]
where the supremum is taken over all cubes $Q \subset \mathbb{R}^n$.

This implies that both Proposition \ref{bmox} and Theorem \ref{thm1}
when $X = L^\infty$ fail because there exist unbounded functions
in ${\rm BMO}$, for example, the logarithmic function.
Indeed, for any $x\in \mathbb{R}^n$, let
\begin{align*}
f(x)=\begin{cases}
\log\log (\frac{1}{|x|}), &|x|\le \frac{1}{e},\\
0, &|x|> \frac{1}{e}.
\end{cases}
\end{align*}
Then it is well known that $f\in {\rm VMO}\subset {\rm BMO}$. However,
$$\|f\|_{{\rm BMO}_{L^\infty}} \sim \sup_Q \operatorname*{ess\,sup}_{x, y \in Q} |f(x) - f(y)|=\infty,$$
which shows that $f\notin {\rm BMO}_{L^\infty}$ and hence $f\notin {\rm VMO}_{L^\infty}$.
Thus, both Proposition \ref{bmox} and Theorem \ref{thm1} when $X = L^\infty$ fail.
\end{remark}

\subsection{Applications to Specific Function Spaces}\label{s4.2}

In this subsection, we use Theorem \ref{thm1} to some specific function spaces
including weighed Lebesgue spaces, variable Lebesgue spaces, mixed-norm Lebesgue spaces,
Morrey spaces, grand Besov--Bourgain--Morrey spaces, Lorentz spaces, and Herz spaces.
All these results, namely Corollaries \ref{coro-wLp}, \ref{coro-vLp}, \ref{coro-mLp},
\ref{coro-morrey}, \ref{coro-gbbm}, \ref{coro-lor}, and \ref{coro-herz},
are new.
Using their definitions, it is easy to verify that all these function spaces are special
case of ball Banach function spaces.
Thus, to apply Theorem \ref{thm1}, it suffices to verify that
the Hardy--Littlewood maximal operator $M$ is bounded on their associated spaces.
To this end, we need the following lemma, which connects
the associated space and the dual space.

\begin{lemma}\label{lem-acn}
Let $X$ be a ball Banach function space.
Denote $X'$ and $X^\ast$ respectively the associate space and the dual space of $X$.
Then $X'=X^\ast$ if and only if $X$ has an absolutely continuous norm,
that is, for any $f\in X$ and any decreasing sequences of measurable sets
$\{E_n\}_{n=1}^\infty$ with $E_n \downarrow 0$ almost everywhere on $\mathbb{R}^n$,
it holds that $\|f\mathbf{1}_{E_n}\|_X \downarrow 0$ as $n\to\infty$.
\end{lemma}

Lemma \ref{lem-acn} is a part of \cite[Lemma 1.7.7]{lyh22};
see also \cite[Proposition 3.15]{ln24}.
So we omit the details here.

Next, we consider specific function spaces.

{\bf Weighted Lebesgue spaces.} \quad
Muckenhoupt \cite{m72} introduced the $A_p$-weight which characterizes
the boundedness of the Hardy-Littlewood maximal operator $M$ on weighted Lebesgue spaces.
From then on, there exist numerous studies on the $A_p$-weight in harmonic analysis.
We refer to the classical monograph of Garc\'ia-Cuerva and Rubio de Francia \cite{gr85}
for a systemic study of weighted theory and related topics.
Recall the definitions of $A_p$-weights and weighed Lebesgue spaces as follows.
\begin{definition}
Let $p\in[1,\infty)$. The \emph{class $A_p$} of Muckenhoupt weights is defined to be the set of all
locally integrable and nonnegative functions $\omega$ on $\mathbb{R}^n$ such that,
when $p\in(1,\infty)$,
\begin{equation*}
[\omega]_{A_p}:=\sup_{\text{balls }B}\left[\fint_B\omega(x)\,dx\right]
\left\{\fint_B\left[\omega(x)\right]^{\frac1{1-p}}\,dx\right\}^{p-1}<\infty
\end{equation*}
and, when $p=1$,
$$
[\omega]_{A_1}:=\sup_{\text{balls }B}\frac1{|B|}\int_B\omega(x)\,dx
\left[\left\|\omega^{-1}\right\|_{L^\infty(B)}\right]<\infty.
$$
Moreover, the \emph{weighted Lebesgue space $L_\omega^p$} is defined
to be the set of all measurable functions $f$ on $\mathbb{R}^n$ such that
$$
\|f\|_{L^p_\omega}:=\left[\int_{\mathbb{R}^n}|f(x)|^p\omega(x)\,dx\right]^\frac1p<\infty.
$$
\end{definition}

\begin{corollary}\label{coro-wLp}
Let $p\in[1,\infty)$ and $\omega\in A_p$.
Then Theorem \ref{thm1} holds with $X$ replaced by $L_\omega^p$.
\end{corollary}

\begin{proof}
It is easy to check that $L_\omega^p$ is a ball Banach function space
having an absolutely continuous norm.
So we only need to verify that $M$ is bounded on $(L_\omega^p)'$.
From \cite[Theorem 2.7.4]{dhhr11},
we deduce that, when $p\in(1,\infty)$ and $\omega\in A_p$,
$$\left(L_\omega^p\right)'=L^{p'}_{\omega^{1-p'}}
\ \text{ and }\ \omega^{1-p'}\in A_{p'}.$$
Recall that, for any $p\in(1,\infty)$,
$M$ is bounded on $L^p_\omega$ if and only if $\omega\in A_p$;
see, for instance, \cite[p.\,137, Theorem 7.3]{Duo}.
Therefore, $M$ is bounded on $(L_\omega^p)'$ for any $p\in(1,\infty)$
and $\omega\in A_p$.

It remains to show the boundedness of $M$ on $(L_\omega^1)'$ for any $w\in A_1$.
By a dual observation, we have
$$(L_\omega^1)'=\left\{f\in\mathscr{M}:\
\|f\|_{(L_\omega^1)'}:=\left\|\frac f\omega\right\|_{L^\infty}<\infty\right\}.$$
Moreover, for any $w\in A_1$ and any ball $B$ centered at $x$,
\begin{align*}
\fint_{B}|f(y)|\,dy
&=\fint_{B}|f(y)|\omega(y)\omega^{-1}(y)\,dy
\le \left\|\frac f\omega\right\|_{L^\infty}\fint_{B}\omega(y)\,dy
\le \left\|\frac f\omega\right\|_{L^\infty}[\omega]_{A_1}\omega(x).
\end{align*}
Taking the supremum over all balls $B$, we then obtain
$$\left\|\frac {M(f)}{\omega}\right\|_{L^\infty}
\le [\omega]_{A_1} \left\|\frac f\omega\right\|_{L^\infty}.$$
This shows that $M$ is bounded on $(L_\omega^1)'$,
which completes the proof of Corollary \ref{coro-wLp}.
\end{proof}

\begin{remark}\label{rem-wLp}
\begin{itemize}
\item[\rm (i)] Corollary \ref{coro-wLp} is new,
even for the unweighted case $\omega=1$.
In this case, for ${\rm VMO}$ and ${\rm CMO}$,
these characterizations  seem to be elementary conclusions
but we did not find explicit statements in the existed literature.

\item[\rm (ii)] One can further show that
$M$ is bounded on $(L_\omega^1)'$ if and only if
$w\in A_1$. Indeed, the ``if'' part is showed in the proof of Corollary \ref{coro-wLp}.
Conversely, the ``only if'' part also holds due to the observation
$$\frac{\fint_{B}\omega(y)\,dy}{\omega(x)}
\le \frac{M(\omega)(x)}{\omega(x)}
\lesssim \frac{\omega(x)}{\omega(x)}=1$$
for any $B$ containing $x$.
\end{itemize}
\end{remark}

{\bf Variable Lebesgue spaces.} \quad
The variable-exponent Lebesgue space emerged from research on
phenomena with spatially varying integrability
(for instance, heterogeneous PDEs and image processing).
Its flexible exponent structure plays a crucial role in
modeling non-uniform behaviors,
and it is now widely applied in harmonic analysis and PDEs.
For more recent progress of variable Lebesgue spaces,
we refer to the monographs of Cruz-Uribe and Fiorenza \cite{cf13}
and Diening et al. \cite{dhhr11}.

\begin{definition}
Let $p(\cdot):\ \mathbb{R}^n\to[0,\infty)$ be a measurable function.
Then the \emph{variable Lebesgue space
$L^{p(\cdot)}$} is defined to be the set of all measurable functions $f$ on $\mathbb{R}^n$ such that
$$
\|f\|_{L^{p(\cdot)}}:=\inf\left\{\lambda\in(0,\infty):\ \int_{\mathbb{R}^n}
\left[\frac{|f(x)|}{\lambda}\right]^{p(x)}\,dx\le1\right\}<\infty.
$$
Moreover, let $\widetilde{p}_-:=\operatorname*{ess\,inf}_{x\in\mathbb{R}^n}\,p(x)$
and $\widetilde p_+:=\operatorname*{ess\,sup}_{x\in\mathbb{R}^n}\,p(x)$.
Furthermore, $p(\cdot)$ is said to be
\emph{globally log-H\"older continuous} if there exist  $p_{\infty}\in\mathbb{R}$
and a positive constant $C$ such that, for any
$x,\ y\in\mathbb{R}^n$,
$$
|p(x)-p(y)|\le C\frac{1}{\log(e+1/|x-y|)}
\ \text{ and }\
|p(x)-p_\infty|\le C\frac{1}{\log(e+|x|)}.
$$
\end{definition}

\begin{corollary}\label{coro-vLp}
Let $p(\cdot):\ \mathbb{R}^n\to(0,\infty)$ be a globally
log-H\"older continuous function satisfying
$1\le\widetilde p_-\leq \widetilde p_+<\infty$.
Then Theorem \ref{thm1} holds with $X$ replaced by $L^{p(\cdot)}$.
\end{corollary}

\begin{proof}
It is easy to check that $L^{p(\cdot)}$ is a ball Banach function space
so long as $\widetilde p_-\ge1$.
So it remains to verify that $M$ is bounded on $[L^{p(\cdot)}]'$.
By \cite[Theorem 3.2.13]{dhhr11}, we obtain
$$[L^{p(\cdot)}]'=L^{p'(\cdot)}
\ \text{ with }\ \frac{1}{p(\cdot)}+\frac{1}{p'(\cdot)}=1.$$
Meanwhile, from \cite[Theorem 3.16]{cf13}, it follows that
$M$ is bounded on $L^{p'(\cdot)}$ so long as
$\widetilde p'_->1$,
which holds due to $\widetilde p'_-=(\widetilde p_+)'$ and $\widetilde p_+<\infty$.
This finishes the proof of Corollary \ref{coro-vLp}.
\end{proof}

\begin{remark}
Corollary \ref{coro-vLp} is new.
\end{remark}

{\bf Mixed-norm Lebesgue spaces.} \quad
In 1961, Benedek and Panzone \cite{BP} studied the mixed-norm Lebesgue space $L^{\vec{p}}$,
which can be traced back to H\"ormander \cite{H1}.
Later on, in 1970, Lizorkin \cite{l70} further developed both the theory of
multipliers of Fourier integrals and estimates of convolutions
in the mixed-norm Lebesgue spaces.
In recent years, the real-variable theory of mixed-norm function spaces
has rapidly been developed to meet the requirements arising in the study of
the boundedness of operators, partial differential equations, and some
other analysis subjects.
We refer to the systematic survey \cite{hy} for more recent progress
on the theory of function spaces with mixed-norm.

\begin{definition}
Let $\vec{p}:=(p_1,\ldots,p_n)\in(0,\infty]^n$.
The \emph{mixed-norm Lebesgue space $L^{\vec{p}}$} is defined
to be the set of all measurable functions $f$ on $\mathbb{R}^n$ such that
$$
\|f\|_{L^{\vec{p}}}:=\left\{\int_{\mathbb{R}}\cdots\left[
\int_{\mathbb{R}}|f(x_1,\ldots,x_n)|^{p_1}\,dx_1\right]
^{\frac{p_2}{p_1}}\cdots\,dx_n\right\}^{\frac{1}{p_n}},
$$
with the usual modifications made when $p_i=\infty$ for some $i\in\{1,\ldots,n\}$, is finite.
\end{definition}

\begin{corollary}\label{coro-mLp}
Let $\vec{p}:=(p_1,\ldots,p_n)\in[1,\infty)^n$.
Then $M$ is bounded on $(L^{\vec{p}})'$ if and only if
$\vec{p}$ satisfies any one of the following conditions:
\begin{itemize}
\item[\rm (i)] $p_i>1$ for any $i\in\{1,\dots,n\}$,
\item[\rm (ii)] $p_i=1$ for any $i\in\{1,\dots,n\}$,
\item[\rm (iii)] there exists some $j\in\{1,\dots,n-1\}$ such that
$1=p_1=\cdots=p_j<\min\{p_{j+1},\dots,p_n\}$.
\end{itemize}
As a consequence, if $\vec{p}$ satisfying (i) or (ii) or (iii),
then Theorem \ref{thm1} holds with $X$ replaced by $L^{\vec{p}}$.
\end{corollary}

\begin{proof}
It is easy to check that $L^{\vec{p}}$ is a ball Banach function space,
whose norm is also an absolutely continuous norm.
By this, the dual theorem of mixed-norm Lebesgue spaces
(see, for instance, \cite[Theorem 1.a]{BP})
and Lemma \ref{lem-acn},
we have
$(L^{\vec{p}})'=L^{\vec{p}'}$ with $\vec{p}':=(p_1',\dots,p_n')$.
From this and the boundedness of $M$ on mixed-norm Lebesgue spaces
(see, for instance, \cite[Lemma 3.5]{HLYY} and \cite[Remark 4.4]{hy}),
it follows that $M$ is bounded on $L^{\vec{p}'}$ if and only if
$\vec{p}$ satisfies (i) or (ii) or (iii).
Using this and Theorem \ref{thm1}, we obtain the desired conclusions,
which completes the proof of Corollary \ref{coro-mLp}.
\end{proof}

\begin{remark}
Corollary \ref{coro-mLp} is new.
Moreover, it should be pointed out that, for the mixed-norm Lebesgue,
we obtain the sufficient and \emph{necessary} condition of $\vec{p}$
such that $M$ is bounded on $X'$ with $X=L^{\vec{p}}$.
While, for other specific function spaces,
except the weighted Lebesgue spaces [see Remark \ref{rem-wLp}(ii)],
we usually only have some sufficient conditions.

\end{remark}

{\bf Morrey spaces.} \quad
Morrey \cite{m38} introduced what is now called Morrey spaces in 1938
and used it to study the local behavior of solutions to
second order elliptic partial differential equations.
In recent decades, there exists an increasing interest in
applications of Morrey spaces to various areas of analysis
such as partial differential equations, potential theory, and harmonic analysis.
For more recent progress of Morrey spaces,
we refer to the monographs of Adams \cite{a15},
Sawano et al. \cite{sfk20i,sfk20ii}, and Yuan et al. \cite{ysy10}.
\begin{definition}\label{def-morrey}
Let $0<q\le p<\infty$.
The \emph{Morrey space $\mathcal{M}_p^q$} is defined
to be the set of all measurable functions $f$ on $\mathbb{R}^n$ such that
\begin{equation*}
\| f \|_{{\mathcal M}^{p}_q}
:=
\sup_{x \in {\mathbb R}^n,\, R\in(0,\infty)}
|B(x,R)|^{\frac{1}{p}-\frac{1}{q}}
\left[
\int_{B(x,R)}|f(y)|^{q}\,dy
\right]^{\frac{1}{q}}<\infty.
\end{equation*}
\end{definition}

\begin{corollary}\label{coro-morrey}
Let $1<q \le p<\infty$.
Then Theorem \ref{thm1} holds with $X$ replaced by ${\mathcal M}^{p}_q$.
\end{corollary}

\begin{proof}
It is easy to verify that ${\mathcal M}^{p}_q$ is a ball Banach function space.
Meanwhile, recall that the associate space of ${\mathcal M}^{p}_q$
is the block space on which $M$ is bounded; see, for instance,
\cite[Theorem 3.1]{ch14}, \cite[Lemma 5.7]{h15}, and \cite[Theorem 4.1]{st15}.
This finishes the proof of Corollary \ref{coro-morrey}.
\end{proof}

\begin{remark}\label{rem-morrey}
\begin{itemize}
\item[\rm (i)]
Corollary \ref{coro-morrey} is new.

\item[\rm (ii)]
The well-known result of Campanato \cite{c64} shows that, when $\alpha\in[-\frac1q,0)$,
the Morrey space $\mathcal{M}^{-1/\alpha}_q$ coincides with
Campanato space $\mathcal{L}_{\alpha,\,L^q}$ in the sense of modulo polynomials.
Meanwhile, Almeida and Samko \cite{as18,as17} introduced (generalized) vanishing Morrey spaces,
which can be regarded as vanishing Campanato spaces for negative $\alpha$,
and then obtained the boundedness of some classical operators
on these vanishing subspaces in \cite{aas20i,aas20ii,a20}.
So one may ask whether it is possible to prove Theorem \ref{thm1} for negative $\alpha$?
However, unlike BMO, Morrey spaces do not have the self-improving property,
that is, $\mathcal{M}^{-1/\alpha}_q$ is not invariant for different $q$
and hence \eqref{LX=L} fails for negative $\alpha$.
\end{itemize}
\end{remark}

{\bf Grand Besov--Bourgain--Morrey spaces.} \quad
In recent years, Morrey-type spaces have also proved useful in harmonic analysis
and PDEs; see, for example, \cite{dn20,hs21,h-23,h24-i,h24-ii,ly13,mst18,s03}.
In particular, Bourgain \cite{b91} introduced a novel function space,
nowadays called the Bourgain--Morrey space, and used it to refine
the classical Stein--Tomas (Strichartz) estimate.
Later on, Masaki \cite{m16-09234} further investigated this space in the full range of indices.
Recently, there exist numerous studies of Bourgain--Morrey spaces on PDEs, such as
the nonlinear Schr\"{o}dinger equation, the Korteweg-de Vries (KdV) equation,
and the Airy equation;
see, for example, \cite{BV07,mvv99,ms18,ms18-2,m16-09234}.

To study the integrability of the Jacobian determinant,
Iwaniec and Sbordone \cite{is92} introduced Grand Lebesgue spaces.
Later, Greco et al. \cite{gis97} further introduced generalized grand Lebesgue spaces
to extend the $p$-harmonic operator to slightly larger spaces than classical Lebesgue spaces.
From then on, grand Lebesgue spaces have been widely studied and used in harmonic analysis
and partial differential equations.
In a very recent article, via combining the structure of
generalized grand Lebesgue spaces in \cite{gis97} and
Besov--Bourgain--Morrey spaces in \cite{zstyy22},
Wan et al. \cite{wyz25} introduced grand Besov--Bourgain--Morrey spaces
and study the nontriviality, the embedding,
and the boundedness of some operators.
We recall the definition of these function spaces as follows.

\begin{definition}\label{def-gBBM}
\begin{itemize}
\item[\rm (i)]
Let $p\in (1,\infty)$, $\theta\in[0,\infty)$, and
$\Omega$ be a measurable subset of ${\mathbb{R}^n}$ with
$|\Omega|\in(0,\infty)$.
The \emph{generalized grand Lebesgue space
$L^{p),\theta}(\Omega)$} is defined
to be the set of all locally integrable functions $f$ on
$\Omega$ such that
\begin{equation*}
\left\|f\right\|_{L^{p),\theta}(\Omega)}:=\sup_{
\varepsilon\in (0,p-1)}
\varepsilon^{\frac{\theta}{p-\varepsilon}}
\left[\frac1{|\Omega|}\int_{\Omega}
\left|f(x)\right|^{p-\varepsilon}\,dx
\right]^{\frac1{p-\varepsilon}}<\infty.
\end{equation*}
Moreover, let $L^{\infty),\theta}(\Omega):=L^\infty(\Omega)$.

\item[\rm (ii)]
Let $0<q\le p\le \infty$ and $r\in(0,\infty]$.
The \emph{Bourgain--Morrey space}
$\mathcal{M}^p_{q,r}$ is defined to be the
set of all $f\in L^q_{{\mathop\mathrm{\,loc\,}}}$ such that
\begin{equation*}
\left\|f\right\|_{\mathcal{M}^p_{q,r}}
:=\left[\sum_{\nu\in{\mathbb Z},m\in{\mathbb Z}^n}
\left\{\left|Q_{\nu,m}\right|^{\frac{1}{p}-
\frac{1}{q}}\left[\int_{Q_{\nu,m}}\left|f(x)\right|^q\,dx\right]
^{\frac1q}
\right\}^r\right]^{\frac{1}{r}},
\end{equation*}
with the usual modifications made when $q=\infty$ and/or $r=\infty$,
is finite.

\item[\rm (iii)]
Let $1<q \le p\le r\le\infty$,
$\tau\in(0,\infty]$, and $\theta\in[0,\infty)$.
The \emph{grand Besov--Bourgain--Morrey space}
$\mathcal{M}\dot{B}_{q),r,\theta}^{p,\tau}$ is defined to be the set
of all $f\in \widetilde{L}^q_{{\mathop\mathrm{\,loc\,}}}
:=\bigcap_{\varepsilon\in (0,q-1)}
L^{q-\varepsilon}_{{\mathop\mathrm{\,loc\,}}}$ such that
\begin{align*}
\|f\|_{\mathcal{M}\dot{B}_{q),r,\theta}^{p,\tau}}
:=\left[\sum_{\nu\in\mathbb{Z}}
\left\{\sum_{m\in\mathbb{Z}^n}
\left[\left|Q_{\nu,m}\right|^\frac{1}{p}
\left\|
f\right\|_{L^{q),\theta}(Q_{v,m})}
\right]^r \right\}^\frac{\tau}{r}\right]^\frac{1}{\tau},
\end{align*}
with the usual modifications made when  $r=\infty$ and/or
$\tau=\infty$, is finite.
\end{itemize}
\end{definition}

Grand Besov--Bourgain--Morrey spaces when $\theta=0$ and $\tau=r$ coincide with Bourgain--Morrey spaces.
Bourgain--Morrey spaces when $r=\infty$ coincide with Morrey spaces.
To be precise, we have the following equivalences:
\begin{itemize}
\item $\mathcal{M}\dot{B}_{q),r,0}^{p,\tau}=:\mathcal{M}\dot{B}_{q,r}^{p,\tau}$,
which is the Besov--Bourgain--Morrey space introduced by Zhao et al. \cite{zstyy22};

\item $\mathcal{M}\dot{B}_{q),r,0}^{p,r}=\mathcal{M}\dot{B}_{q,r}^{p,r}={\mathcal M}^p_{q,r}$;

\item $\mathcal{M}\dot{B}_{q),\infty,0}^{p,\infty}=\mathcal{M}\dot{B}_{q,\infty}^{p,\infty}={\mathcal M}^p_{q,\infty}=\mathcal{M}_q^p$,
which is the Morrey space in Definition \ref{def-morrey}.
\end{itemize}

\begin{corollary}\label{coro-gbbm}
Let $p,q,r,\tau$, and $\theta$ satisfy
either of the following nontrivial conditions:
\begin{enumerate}
\item[$\mathrm{(i)}$] $1<q\leq  p<r=\tau =\infty$ and $\theta\in[0,\infty)$.
\item[$\mathrm{(ii)}$] $1<q< p\le r<\tau= \infty$
and $\theta\in[0,\infty)$.
\item[$\mathrm{(iii)}$] $1<q< p<r\le \infty$, $\tau\in[1,\infty)$,
and $\theta\in[0,\infty)$.
\item[$\mathrm{(iv)}$] $1<q= p<r< \infty$, $\tau\in[1,\infty)$,
and $\theta\in(\frac{q}{r},\infty)$.
\item[$\mathrm{(v)}$] $1<q= p<r= \infty$, $\tau\in[1,\infty)$,
and $\theta\in(\frac{q}{\tau},\infty)$.
\item[$\mathrm{(vi)}$] $1<q= p\le r<\tau= \infty$
and $\theta\in (\frac{q}{r},\infty)$.
\end{enumerate}
Then Theorem \ref{thm1} holds with $X$ replaced by $\mathcal{M}\dot{B}_{q),r,\theta}^{p,\tau}$.
\end{corollary}

\begin{proof}
The boundedness of $M$ on $(\mathcal{M}\dot{B}_{q),r,\theta}^{p,\tau})'$
was established in \cite[Theorem 4.12]{wyz25}.
\end{proof}

\begin{remark}
Corollary \ref{coro-gbbm} is new.
\end{remark}

{\bf Lorentz spaces.} \quad
Recall that, for any $p,q\in(0,\infty]$,
the \emph{Lorentz space $L^{p,q} $}, originally studied by Lorentz \cite{Lor50,Lor51},
is defined to be the set of all
$f\in\mathscr{M} $ such that
\begin{equation*}
\|f\|_{L^{p,q} }
:=\left\{\int_0^{\infty}
\left[t^{\frac{1}{p}}\sup_{\{E\subset\mathbb{R}^n:\,|E|\ge t\}}\frac{1}{|E|}\int_E |f(x)|\,dx\right]^q
\frac{\,dt}{t}\right\}^{\frac{1}{q}},
\end{equation*}
with the usual modification made when $q=\infty$, is finite.

\begin{corollary}\label{coro-lor}
Let $p,q\in(1,\infty)$.
Then Theorem \ref{thm1} holds with $X$ replaced by $L^{p,q}$.
\end{corollary}
\begin{proof}
It is easy to verify that $L^{p,q}$ is a ball Banach function space.
Meanwhile, the boundedness of $M$ on $(L^{p,q})'$ was established
in the proof of \cite[Theorem 5.18]{zyy23},
which completes the proof of Corollary \ref{coro-lor}.
\end{proof}

\begin{remark}
Corollary \ref{coro-lor} is new.
\end{remark}

{\bf Herz spaces.} \quad
In 1968, Herz \cite{herz} introduced the classical Herz spaces
and used it to study the Bernstein theorem on absolutely
convergent Fourier transforms.
Recently, Rafeiro and Samko \cite{RS20}
introduced the local and the global generalized Herz spaces,
which respectively generalize the classical
Herz spaces and generalized Morrey type spaces.
For more studies on the (generalized) Herz spaces,
we refer to the monographs \cite{lyh08,lyh22}.

\begin{definition}
Let $\alpha\in\mathbb{R}$ and $p,q\in(0,\infty]$.
\begin{itemize}
\item[\rm (i)] The \emph{homogeneous Herz space $\dot{K}^{\alpha,p}_q$} is defined to be
the set of all locally integrable functions $f$ on $\mathbb{R}^n\setminus\{\mathbf{0}\}$
such that
$$\|f\|_{\dot{K}^{\alpha,p}_q}:=\left[
\sum_{k\in\mathbb{Z}} 2^{k \alpha p}
\left\|f\mathbf{1}_{B(\mathbf{0},2^k)\setminus B(\mathbf{0},2^{k-1})}
\right\|_{L^q}^p\right]^{\frac1p},$$
with the usual modification made when $p=\infty$, is finite.

\item[\rm (ii)] The \emph{non-homogeneous Herz space ${K}^{\alpha,p}_q$} is defined to be
the set of all locally integrable functions $f$ on $\mathbb{R}^n$
such that
$$\|f\|_{{K}^{\alpha,p}_q}:=\left[
\left\|f\mathbf{1}_{B(\mathbf{0},1)}\right\|_{L^q}^p
+ \sum_{k=1}^\infty 2^{k \alpha p}
\left\|f\mathbf{1}_{B(\mathbf{0},2^k)\setminus B(\mathbf{0},2^{k-1})}
\right\|_{L^q}^p\right]^{\frac1p},$$
with the usual modification made when $p=\infty$, is finite.
\end{itemize}
\end{definition}

\begin{corollary}\label{coro-herz}
Let $p, q\in[1,\infty)$ and $\alpha\in(-n(1-\frac1q),\frac nq)$.
Then Theorem \ref{thm1} holds with $X$ replaced by
$\dot{K}^{\alpha,p}_q$ or ${K}^{\alpha,p}_q$.
\end{corollary}
\begin{proof}
It is easy to verify that both $\dot{K}^{\alpha,p}_q$ and ${K}^{\alpha,p}_q$
are ball Banach function spaces.
By the dual property of Herz spaces
(see, for instance, \cite[Theorem 1.7.9]{lyh22} and also
\cite[pp.\,8-9, Corollraies 1.2.1 and 1.2.2]{lyh08}),
we have
$(\dot{K}^{\alpha,p}_q)'=\dot{K}^{-\alpha,p'}_{q'}$
and
$({K}^{\alpha,p}_q)'={K}^{-\alpha,p'}_{q'}.$
Combining this and the boundedness of $M$ on Herz spaces
(see, for instance, \cite[p.\, 131, Theorem 5.1.1]{lyh08}
and also \cite[p.\,81, Corollary 1.5.4]{lyh22})
then finishes the proof of Corollary \ref{coro-herz}.
\end{proof}

\begin{remark}
\begin{itemize}
\item[\rm (i)]
Corollary \ref{coro-herz} is new.

\item[\rm (ii)]
One can similarly show that Theorem \ref{thm1} holds with $X$ replaced by
the local generalized Herz space of Rafeiro and Samko \cite{RS20};
see, for instance, the proof of \cite[Theorem 4.15]{zyy24}.
However, for the global generalized Herz space, it is unclear so far
because we do not know its associate space.
\end{itemize}
\end{remark}

\begin{appendices}

\section{Convolutions and Fractional Integral Commutators}\label{a}

Here as an appendix, by presenting two examples, we show that
vanishing Campanato spaces arise naturally as the ranges of
integral-type operators. Section \ref{s5.2} considers the convolution,
while Section \ref{s5.1} deals with commutators.

\subsection{Convolutions\label{s5.2}}

Let $\alpha\in(0,\infty)$ and $\varphi \in L^1$. Consider the mapping
$f \in \mathcal{L}_\alpha \mapsto \varphi * f \in \mathcal{L}_\alpha.$
It is straightforward to verify that, for any $f \in \mathcal{L}_\alpha$
and $\varphi \in L^1$,
\begin{align}\label{a*f}
\|\varphi * f\|_{\mathcal{L}_\alpha}
\le \|\varphi\|_{L^1} \|f\|_{\mathcal{L}_\alpha}.
\end{align}
About the range of this convolution mapping, we have the following property:
\begin{proposition}\label{pa1}
Let $\alpha\in(0,\infty)$ and let
$f \in \mathcal{L}_\alpha$ and $\varphi \in L^1$.
Then
$\varphi * f \in \mathrm{V}\mathcal{L}_\alpha$.
\end{proposition}

\begin{proof}
Since $C_{\rm c}^\infty$ is dense in $L^1$,
we may assume without loss of generality that $\varphi \in C_{\rm c}^\infty$
due to \eqref{a*f}.
Then ${\rm supp\,}(\varphi) \subset B(\mathbf{0},R)$ for some $R\in(0,\infty)$.
Let $N:=\lfloor \alpha \rfloor+1$.
Write $\widetilde{f}_x(y):=f(x - y)$ for any $x,y\in\mathbb{R}^n$
and note that
\begin{align*}
|\nabla^N(\varphi * f)(x)|
&= |(\nabla^N \varphi) * f(x)|
= \left|\int_{\mathbb{R}^n} \nabla^N a(y)
\left[\widetilde{f}_x(y) - P_{B(\mathbf{0},R)}^{\lfloor \alpha \rfloor}(\widetilde{f}_x)(y)\right] \, dy\right|\\
&\le \int_{B(\mathbf{0},R)} \left\|\nabla^N a\right\|_{L^\infty}
\left[\widetilde{f}_x(y) - P_{B(\mathbf{0},R)}^{\lfloor \alpha \rfloor}(\widetilde{f}_x)(y)\right]\, dy\\
&\lesssim \left\|\nabla^N a\right\|_{L^\infty}|B(\mathbf{0},R)|^{1+\frac{\alpha}{n}}
\|\widetilde{f}_x\|_{\mathcal{L}_\alpha}
\sim \|f\|_{\mathcal{L}_\alpha},
\end{align*}
which implies that
$\nabla^N (\varphi * f) \in (L^\infty)^{Nn}.$
Therefore,
$\varphi * f \in \left(\mathcal{L}_N \cap \mathcal{L}_\alpha\right)
\subset \mathrm{V}\mathcal{L}_\alpha$,
which completes the proof of Proposition \ref{pa1}.
\end{proof}

\subsection{Commutators}
\label{s5.1}

Let $1 \le q \le p < \infty$ and ${\mathcal M}^{p}_q$ be
the Morrey space as in Definition \ref{def-morrey}.
Very recently, the compactness of commutators in Morrey space
was investigated by Hakim el al. \cite{hst25}.
We consider the fractional case in this subsection.

Let $\alpha\in(0,n)$.
The \emph{fractional integral operator} \( I_\alpha \) is defined by setting,
for any suitable function \( f \) and any $x\in\mathbb{R}^n$
\[
I_\alpha f(x) :=
\int_{\mathbb{R}^n}
\frac{f(y)}{|x-y|^{n-\alpha}}\, dy.
\]
Observe that $I_\alpha$ is well-defined on
${\mathcal M}^{p}_q$ with $p\in(1,\infty)$.
Indeed, let $f\in {\mathcal M}^{p}_q$
and $B$ be a ball centered at $x$.
For the local part $I_\alpha(f\chi_B)(x)$,
it is well-defined due to the locally integrability of $|\cdot|^{-(n-\alpha)}$.
For the non-local part  $I_\alpha(f\chi_{B^\complement})(x)$,
it suffices to use H\"older's inequality and then dominate it
by the maximal function which is bounded on Morrey spaces due to
\cite[Theorem 1]{cf87}.

Moreover, let $b \in {\rm Lip}_\beta$ with
$0<\beta<\alpha+\beta<n.$
The \emph{commutator} \( [b, I_\alpha] \) is defined by setting,
for any suitable function $f$ and any $x\in\mathbb{R}^n$,
\[
[b, I_\alpha]f(x) :=
b(x)I_\alpha(f)(x)-I_\alpha(bf)(x) =
\int_{\mathbb{R}^n}
\frac{b(x) - b(y)}{|x-y|^{n-\alpha}}
f(y) \, dy.
\]
It is also well-defined on Morrey spaces because
\[\left|\frac{b(x) - b(y)}{|x-y|^{n-\alpha}}\right|
\lesssim \frac{1}{|x-y|^{n-(\alpha+\beta)}}
\text{ for any }x, y\in\mathbb{R}^n \text{ with }x\neq y\]
and hence we can repeat the argument as above.

\begin{lemma}\label{prop:250505-1}
Suppose that $\alpha\in(0,n)$, $\beta\in(0,1)$, $1 < q \leq p < \infty$,
and $b \in {\rm Lip}_\beta$.
\begin{enumerate}
    \item[\rm(i)] If
    $
    \frac{n}{p} = \alpha + \beta,
    $
    then $[b,  I_\alpha ]$ maps ${\mathcal M}^p_q$ to ${\rm BMO}$. The operator norm of $[b,  I_\alpha ]$ satisfies
    \[
    \left\|[b,  I_\alpha ]\right\|_{{\mathcal M}^p_q
     \to {\rm BMO}} \lesssim
     \|b\|_{{\rm Lip}_\beta}.
    \]
    \item[\rm(ii)] If
    $
    \gamma = \alpha + \beta - \frac{n}{p} \in (0,1),
    $
    then $[b,  I_\alpha ]$ maps ${\mathcal M}^p_q$ to ${\rm Lip}_\gamma$. The operator norm of $[b,  I_\alpha ]$ satisfies
    \[
    \left\|[b,  I_\alpha ]\right\|_{{\mathcal M}^p_q
     \to {\rm Lip}_\gamma}
     \lesssim \|b\|_{{\rm Lip}_\beta}.
    \]

    \item[\rm(iii)] If $\alpha$, $\beta$, $s$ and $t$ satisfy
    $\alpha + \beta<\frac{n}{p},$ $\frac{1}{s} = \frac{1}{p} - \frac{\alpha}{n},$
    $1<s\le t<\infty$, and $\frac{t}{s} = \frac{q}{p},$
    then $[b,  I_\alpha ]$ maps ${\mathcal M}^p_q$ to ${\mathcal M}^s_t$. The operator norm of $[b,  I_\alpha ]$ satisfies
    $\|[b,  I_\alpha ]\|_{{\mathcal M}^p_q
     \to{\mathcal M}^s_t}
     \lesssim \|b\|_{{\rm Lip}_\beta}.$
\end{enumerate}
\end{lemma}

\begin{proof}
We only prove the first assertion,
as the second and the third can be proven similarly.
Let $x, z \in \mathbb{R}^n$. We decompose:
\begin{align*}
&[b,  I_\alpha ]f(x) - [b,  I_\alpha ]f(z) \\
&\quad= \int_{\mathbb{R}^n}
\left[
\frac{b(x) - b(y)}{|x-y|^{n-\alpha}} - \frac{b(z) - b(y)}{|z-y|^{n-\alpha}}
\right]
f(y) \, dy \\
&\quad= \int_{\mathbb{R}^n}
\left[
\frac{b(x) - b(y)}{|x-y|^{n-\alpha}} - \frac{b(x) - b(y)}{|z-y|^{n-\alpha}}
\right]
f(y) \, dy
+ \int_{\mathbb{R}^n}
\frac{b(x) - b(z)}{|z-y|^{n-\alpha}}
f(y) \, dy.
\end{align*}

Let $Q:=B(c(Q),r(Q))$ be a ball. Suppose $x, z \in Q$ and $y \in \mathbb{R}^n \setminus 3Q$. Then
\begin{align*}
\left|
\frac{b(x) - b(y)}{|x-y|^{n-\alpha}} - \frac{b(x) - b(y)}{|z-y|^{n-\alpha}}
\right|
\lesssim \frac{r(Q)}{|x-y|^{n-\alpha+1-\beta}}.
\end{align*}
Therefore
\begin{align*}
&|[b,  I_\alpha ][\chi_{\mathbb{R}^n \setminus 3Q}f](x) - [b,  I_\alpha ][\chi_{\mathbb{R}^n \setminus 3Q}f](z)| \\
&\quad\lesssim \int_{\mathbb{R}^n \setminus 3Q}
\left[
\frac{r(Q)}{|x-y|^{n-\alpha+1-\beta}}
+ \frac{[r(Q)]^\beta}{|x-y|^{n-\alpha}}
\right]
|f(y)| \, dy \\
&\quad\lesssim \int_{\mathbb{R}^n \setminus 3Q} \frac{[r(Q)]^\beta}{|x-y|^{n-\alpha}} |f(y)| \, dy \\
&\quad\lesssim [r(Q)]^\beta \int_{3r(Q)}^\infty
\left[
\frac{1}{\ell^{n-\alpha+1}} \int_{B(c(Q), \ell)} |f(y)| \, dy
\right]
d\ell \\
&\quad\lesssim [r(Q)]^\beta \int_{3r(Q)}^\infty
\frac{1}{\ell^{\frac{n}{p}-\alpha+1}}
d\ell\, \|f\|_{{\mathcal M}^p_q}
\lesssim \|f\|_{{\mathcal M}^p_q}.
\end{align*}

Meanwhile,
\begin{align*}
&\fint_{Q \times Q}
|[b,  I_\alpha ][\chi_{3Q}f](x) - [b,  I_\alpha ][\chi_{3Q}f](z)| \, dx dz\\
&\quad\lesssim \fint_{Q \times Q}\int_{3Q}
\left[\frac{|b(x) - b(y)|}{|x-y|^{n-\alpha}}
+\frac{|b(z) - b(y)|}{|z-y|^{n-\alpha}}\right]|f(y)|\, dydxdz \\
&\quad\lesssim \|b\|_{{\rm Lip}_\beta}\fint_{Q}\int_{3Q}
|x-y|^{\alpha+\beta-n}|f(y)|\, dydx \\
&\quad\lesssim \|b\|_{{\rm Lip}_\beta}\fint_{3Q}
\left[\int_{B(y,4r(Q))}
|x-y|^{\alpha+\beta-n}\,dx\right]|f(y)|\, dy \\
&\quad\sim\|b\|_{{\rm Lip}_\beta}
[r(Q)]^{\alpha+\beta} \fint_{3Q} |f(x)| \, dx
\lesssim\|b\|_{{\rm Lip}_\beta}
\|f\|_{{\mathcal M}^p_q}.
\end{align*}

Thus, $[b,  I_\alpha ]$ is bounded from ${\mathcal M}^p_q$ to ${\rm BMO}$,
which completes the proof of Lemma \ref{prop:250505-1}.
\end{proof}

We present a useful criterion for the compactness
of a set
$B \subset {\rm Lip}_\alpha$.\begin{lemma}\label{lem:250426-1}
Let \( 1 < q \leq p < \infty \) and \( 0 \leq \alpha < \beta < 1 \).
Suppose that a set \( B \subset \mathrm{Lip}_\alpha \) satisfies the following conditions{\rm:}
\begin{enumerate}
\item[\rm (i)] \( B \) is bounded in \( \mathcal{M}^p_q \),
\item[\rm (ii)] \( B \) is bounded in \( \mathrm{Lip}_\beta \),
\item[\rm (iii)] for any test function \( \varphi \in C^\infty_{\rm c} \),
\begin{equation}\label{eq:250426-1}
\lim_{R \to \infty}
\sup_{f \in B}
\sup_{x \in \mathbb{R}^n \setminus B(0, R)}
|\varphi * f(x)| = 0.
\end{equation}
\end{enumerate}
Then \( B \) is relatively compact in \( \mathrm{Lip}_\alpha \).
\end{lemma}
\begin{proof}
By normalization, we may assume
\[
\sup_{f \in B}\left(\|f\|_{{\rm Lip}_\beta}
+\|f\|_{{\mathcal M}^p_q}\right)<\infty.
\]
Let $k \in C^\infty_{\rm c}$ satisfy
$\chi_{B(0,1)} \le k \le \chi_{B(0,2)}$  on $\mathbb{R}^n.$
Suppose that we have a sequence
$\{f_j\}_{j=1}^\infty$ in $B$.
Since
$({\mathcal M}^p_q \cap {\rm Lip}_{\beta})
\subset
{\rm Lip}_{\alpha},
$
we find that
$\{f_j\}_{j=1}^\infty$ is a bounded sequence in
${\rm Lip}_{\alpha}$.
Let $\varphi=\Delta k.$
Use the diagonal argument and (\ref{eq:250426-1})
to take a subsequence
$\{f_{j_k}\}_{k=1}^\infty$ such that
$\{2^{l n}\varphi(2^l\cdot)*f_{j_k}\}_{k=1}^\infty$
is uniformly convergent to a function
for each $l \in {\mathbb N}$.
Let $j_1 \le j_0$.
By \cite[Chapter 2 and Section 5.2]{t83}, we have
${\rm Lip}_\beta \sim \dot{B}^\beta_{\infty,\infty},$
${\rm Lip}_\alpha
\sim \dot{B}^\alpha_{\infty,\infty},$ and
${\mathcal M}^p_q \hookrightarrow
\dot{B}^{-\frac{n}{p}}_{\infty,\infty},$
and hence
\begin{align*}
\|g_k-g_{k'}\|_{{\rm Lip}_\alpha}
&\sim
\sup_{j \ge j_0}
2^{j \alpha}\|2^{j n}\varphi(2^j\cdot)*(g_k-g_{k'})\|_{L^\infty}+
\sup_{j_1<j<j_0}
2^{j \alpha}\|2^{j n}\varphi(2^j\cdot)*(g_k-g_{k'})\|_{L^\infty}\\
&\quad+
\sup_{j \le j_1}
2^{j \alpha}\|2^{j n}\varphi(2^j\cdot)*(g_k-g_{k'})\|_{L^\infty}\\
&\le
2^{j_0(\alpha-\beta)}
\sup_{j \ge j_0}
2^{j \beta}\|2^{j n}\varphi(2^j\cdot)*(g_k-g_{k'})\|_{L^\infty}\\
&\quad+ \sup_{j_1<j<j_0}
2^{j \alpha}\|2^{j n}\varphi(2^j\cdot)*(g_k-g_{k'})\|_{L^\infty}\\
&\quad+
2^{j_1(\alpha+\frac{n}{p})}
\sup_{j \le j_1}
2^{-\frac{j n}{p}}\|2^{j n}\varphi(2^j\cdot)*(g_k-g_{k'})\|_{L^\infty}\\
&\le
2^{j_0(\alpha-\beta)}+
\sup_{j_1<j<j_0}
2^{j \alpha}\|2^{j n}\varphi(2^j\cdot)*(g_k-g_{k'})\|_{L^\infty}
+2^{j_1(\alpha+\frac{n}{p})}.
\end{align*}
Letting $k,k' \to \infty$, we obtain
\[
\limsup_{k,k' \to \infty}\|g_k-g_{k'}\|_{{\rm Lip}_\alpha}
\lesssim
2^{j_0(\alpha-\beta)}
+
2^{j_1(\alpha+\frac{n}{p})}.
\]
Letting $j_0 \to \infty$ and $j_1 \to -\infty$, we conclude
\[
\limsup_{k,k' \to \infty}\|g_k-g_{k'}\|_{{\rm Lip}_\alpha}=0.
\]
Thus, $B$ is compact.
This finishes the proof of Lemma \ref{lem:250426-1}.
\end{proof}

Using Lemmas \ref{prop:250505-1} and \ref{lem:250426-1}, we obtain the following compactness on Morrey spaces.
Recall that ${\rm CMO}={\rm C}\mathcal{L}_0$,
${\rm CMO}_\gamma={\rm C}\mathcal{L}_\gamma$ with $\gamma\in(0,1)$,
${\rm VMO}={\rm V}\mathcal{L}_0$,
and ${\rm XMO}={\rm X}\mathcal{L}_0$.

\begin{proposition}\label{prop-a.4}
Suppose that $0 < \alpha < n$, $0 < \beta < 1$, and $1 < q \leq p < \infty$. Assume that $b \in C^\infty_{\rm c}$.
\begin{enumerate}
    \item[\rm(i)] If $\frac{n}{p} = \alpha + \beta,$
    then $[b,  I_\alpha ]$ maps ${\mathcal M}^p_q$ compactly to ${\rm CMO}$.
    \item[\rm(ii)] If $
    \gamma = \alpha + \beta - \frac{n}{p} \in (0,1),$
    then $[b,  I_\alpha ]$ maps ${\mathcal M}^p_q$ compactly to ${\rm CMO}_\gamma$.
    \item[\rm(iii)] If $\alpha$, $\beta$, $s$ and $t$ satisfy
    $\alpha + \beta<\frac{n}{p},$ $\frac{1}{s} = \frac{1}{p} - \frac{\alpha}{n},$
    $1<s\le t<\infty$, and $\frac{t}{s} = \frac{q}{p},$
    then $[b,  I_\alpha ]$ maps ${\mathcal M}^p_q$
    compactly to ${\mathcal M}^s_t$.
\end{enumerate}
\end{proposition}

\begin{proof}
Due to the similarity, we only deal with (i). Let $B$ be the unit ball of ${\mathcal M}^p_q$. It suffices to show that
${\mathcal A} =\{[b,  I_\alpha ]f : f \in B\}$
is relatively compact in ${\rm BMO}$.

Since ${\mathcal A}$ is a bounded set in ${\rm Lip}_\varepsilon$ if $\varepsilon \in (0, 1-\beta)$ [see Proposition \ref{prop:250505-1}(ii)], it follows that ${\mathcal A} \subset {\rm VMO}$.

To show (ii), we claim that
\begin{equation}\label{eq:250426-2}
\sup_{f \in B} \left\|
\mathbf{1}_{\mathbb{R}^n \setminus B(0,R)}[b,  I_\alpha ]f \right\|_{L^\infty}
\lesssim \frac{1}{R^{\frac{n}{p}-\alpha}}
\end{equation}
if $R\in(0,\infty)$ satisfies ${\rm supp}(b) \subset B(0,R/2)$.

Indeed, for any $f\in B$ and $x\in\mathbb{R}^n\setminus B(0,R)$, we note that
\[
\left|[b,  I_\alpha ]f(x)\right|
=
\left|\int_{{\mathbb R}^n}\frac{-b(y)}{|x-y|^{n-\alpha}}f(y)\,dy\right|.
\]
For I, by H\"older's inequality, we have
\begin{align*}
\left|[b,  I_\alpha ]f(x)\right|
&\le\|b\|_{L^\infty}\int_{B(0,R/2)}\frac{1}
{|x-y|^{n-\alpha}}|f(y)|\,dy\\
&\lesssim\|b\|_{L^\infty}\frac1{R^{n-\alpha}}
\int_{B(0,R/2)}|f(y)|\,dy\\
&\lesssim\|b\|_{L^\infty}\frac1{R^{n-\alpha}}
\left[\int_{B(0,R/2)}|f(y)|^q\,dy\right]^{1/q}R^{n/q'}\\
&\lesssim\|b\|_{L^\infty}\frac1{R^{n-\alpha}}
\|f\|_{{\mathcal M}^p_q}
R^{\frac{n}{q}-\frac{n}{p}}R^{n/q'}\lesssim \|b\|_{L^\infty}
\|f\|_{{\mathcal M}^p_q}\frac{1}{R^{\frac{n}{p}-\alpha}},
\end{align*}
which shows the claim.
This claim further implies ${\mathcal A} \subset {\rm XMO}$.

Finally, by Proposition \ref{prop:250505-1}(iii), we conlude that $[b, I_\alpha ]$
maps ${\mathcal M}^p_q$ boundedly to ${\mathcal M}^s_t$ if $1 < t \leq s < \infty$,
$\frac{1}{s} = \frac{1}{p} - \frac{\alpha}{n},$ and $\frac{t}{s} = \frac{q}{p}.$
This means that
\[
|Q|^{-1}\left\|\chi_Q[b,  I_\alpha ]f\right\|_{L^1}
\le
|Q|^{-\frac1t}\left\|\chi_Q[b,  I_\alpha ]f\right\|_{L^t}
\lesssim
|Q|^{-\frac1s}
\|f\|_{{\mathcal M}^p_q}
\lesssim
|Q|^{-\frac1s}.
\]
Thus, ${\mathcal A} \subset {\rm CMO}$.

Since
(\ref{eq:250426-1}) holds
for $B={\mathcal A}$ due to (\ref{eq:250426-2}),
we find that ${\mathcal A}$ is compact due to Lemma \ref{lem:250426-1}.
Thus, $[b,  I_\alpha ]$ is compact,
which completes the proof of Proposition \ref{prop-a.4}.
\end{proof}
\end{appendices}

\smallskip

\noindent\textbf{Acknowledgements}\quad
Jin Tao would like to
thank Jingsong Sun for some useful discussions on Lemma \ref{const},
and to thank Chenfeng Zhu, Yinqin Li, and Yirui Zhao for some useful discussions
on the applications to specific function spaces.

\bigskip

\noindent Xing Fu and Jin Tao (Corresponding author)

\medskip

\noindent Hubei Key Laboratory of Applied Mathematics,
Key Laboratory of Intelligent Sensing System and Security (Ministry of Education of China),
Faculty of Mathematics and Statistics,
Hubei University, Wuhan 430062, The People's Republic of China

\smallskip

\noindent {\it E-mails}:
\texttt{xingfu@hubu.edu.cn} (X. Fu)

\noindent\phantom{{\it E-mails:}}
\texttt{jintao@hubu.edu.cn} (J. Tao)

\bigskip

\noindent Yoshihiro Sawano

\medskip

\noindent Department of Mathematics, Chuo University, Tokyo, 112-8551, Japan

\smallskip

\noindent {\it E-mail}:
\texttt{yoshihiro-sawano@celery.ocn.ne.jp}

\bigskip

\noindent  Dachun Yang

\medskip

\noindent  Laboratory of Mathematics and Complex Systems
(Ministry of Education of China),
School of Mathematical Sciences, Beijing Normal University,
Beijing 100875, The People's Republic of China

\smallskip

\noindent {\it E-mail}:
\texttt{dcyang@bnu.edu.cn}
\end{document}